\documentclass[9pt,leqno,a4paper]{article}
\usepackage{amsfonts}
\usepackage{amsmath, amsfonts, amsthm, amssymb, amscd}
\usepackage[mathcal]{euscript}
\usepackage{mathrsfs}
\usepackage{dsfont}

\newtheorem{theorem}{Theorem}[section]
\newtheorem{proposition}[theorem]{Proposition}
\newtheorem{lemma}[theorem]{Lemma}

\newtheorem{remark}[theorem]{Remark}

\numberwithin{equation}{section}

\usepackage{a4wide}
\usepackage{multirow}
\usepackage{tabularx}
\usepackage{lscape}
\newcommand \Kcal {\mathcal K}
\newcommand \Hcal {\mathcal H}

\newcommand \del \partial
\newcommand \delu {\underline{\del}}
\newcommand \Tu {\underline{T}}

\newcommand \minu{\underline{m}}
\newcommand \Pu{\underline{P}}
\newcommand \Psiu{\underline{\Psi}}
\newcommand \Phiu{\underline{\Phi}}

\newcommand \RR{\mathbb{R}}

\newcommand {\vep}{\varepsilon}

\makeatletter
\def\hlinew#1{%
  \noalign{\ifnum0=`}\fi\hrule \@height #1 \futurelet
   \reserved@a\@xhline}
\makeatother

\let\oldmarginpar\marginpar
\renewcommand\marginpar[1]{\-\oldmarginpar[\raggedleft\footnotesize #1]%
{\raggedright\footnotesize #1}}

\begin{document}
\setcounter{footnote}{-1}
\title{Global solutions of non-linear wave-Klein-Gordon system in two space dimension:\\ semi-linear interactions \footnote{The present work belongs to a research project ``Global stability of quasilinear wave-Klein-Gordon system in
$2 + 1$ space-time dimension'' (11601414), supported by NSFC.}}
\author{Yue MA\footnote{School of Mathematics and Statistics, Xi'an Jiaotong University, Xi'an, Shaanxi 710049, P.R. China}}

\maketitle
\begin{abstract}
In this work we consider the problem of global existence of small regular solutions to a type nonlinear wave-Klein-Gordon system with semi-linear interactions in two spatial dimension. We develop some new techniques on both wave equations and Klein-Gordon equations in order to get sufficient decay rates when energies are not uniformly bounded. These techniques are compatible with those introduced in previous work on two spatial-dimensional quasi-linear wave-Klein-Gordon systems, and can be applied in much general cases.
\end{abstract}
%\tableofcontents
\section{Introduction}\label{sec intro}
\subsection{Objective}
The present work belongs to a research program in which we try to understand the global existence of small regular solution associated to a nonlinear wave-Klein-Gordon system with the presence of different nonlinear terms in $2+1$ space-time dimensions. In this article, we will develop some new techniques in the hyperboloidal foliation framework which can supply more precise decay bounds on wave and Klein-Gordon equations. Then equipped with these tools, we regard the following system
\begin{equation}\label{eq main}
\left\{
\aligned
&\Box u = v\left(A^{\alpha}\del_{\alpha}v + R v\right) + Q^{\alpha\beta}\del_{\alpha}v\del_{\beta}v
\\
&\Box v + v = P^{\alpha\beta}\del_{\alpha}u\del_{\beta}u
\endaligned
\right.
\end{equation}
where in the wave equation presents the quadratic semi-linear terms of Klein-Gordon component while in the Klein-Gordon equation presents the semi-linear terms of wave component. The $A^{\alpha}, R$ and $Q^{\alpha\beta}$ are constants and the quadratic form $P^{\alpha\beta}$ is supposed to be a {\sl null} form, that is,
\begin{equation}\label{eq 0 null}
P^{\alpha\beta}\xi_{\alpha}\xi_{\beta} = 0,\quad \forall \xi\in \RR^3,\quad \xi_0^2-\xi_1^2-\xi_2^2 = 0.
\end{equation}

Such a choice is because that in previous work, the quasi-linear terms and semi-linear terms of Klein-Gordon component coupled in Klein-Gordon equation is treated in \cite{M2}, semi-linear terms of wave component coupled in wave equation (when in scalar case) can be eliminated by a algebraic transformation. So now we regard the interactions of the two components: the semi-linear terms of one component coupled in the other's equation.

The techniques to be developed in the present work is compatible with those developed in \cite{LM1} and \cite{M1}, thus can be applied in the following system:
$$
\left\{
\aligned
\Box u = &(P_1^{\alpha\beta\gamma}\del_{\gamma}u + Q_1^{\alpha\beta\gamma}\del_{\gamma}v + R_1^{\alpha\beta}v)\del_{\alpha}\del_{\beta}u
\\
&+ N_1^{\alpha\beta}\del_{\alpha}u\del_{\beta}u + v\left(A_1^{\alpha}\del_{\alpha}v + R_1 v\right) + B_1^{\alpha\beta}\del_{\alpha}v\del_{\beta}v,
\\
\Box v + v =&(P_2^{\alpha\beta\gamma}\del_{\gamma}u + Q_2^{\alpha\beta\gamma}\del_{\gamma}v + R_2^{\alpha\beta}v)\del_{\alpha}\del_{\beta}v
\\
&+ N_2^{\alpha\beta}\del_{\alpha}u\del_{\beta}u + v\left(A_2^{\alpha}\del_{\alpha}v + R_2 v\right) + B_2^{\alpha\beta}\del_{\alpha}v\del_{\beta}v.
\endaligned
\right.
$$
But for the clarity we limit on the case \eqref{eq main}.

The problem of global existence of small regular solution to nonlinear hyperbolic systems has attracted lots of attention of the mathematical community. We recall the pioneer works belong to S. Klainerman in $3+1$ space-time dimensions, where he introduce the vector field method on wave equations (\cite{Kl1}) and Klein-Gordon equations (in \cite{Kl2}). This type of method is then largely extended in many contexts.

In $2+1$ space-time dimensions, the global existence result becomes more delicate, because the decay rate of the homogeneous linear wave equation and homogeneous linear Klein-Gordon is significantly weaker than the $3+1$ case. But equipped with some carefully developed technical tools, people also managed to prove the global existence in many cases, we recall especially the work of S. Alinhac (\cite{A1} and \cite{A2}) on the quasi-linear wave equations and the work of J-M. Delort et al. (\cite{DFX}) on quasi-linear Klein-Gordon equations.

When we consider the system composed by wave and Klein-Gordon equations, the situation becomes more complicated. Because one of the conformal Killing vector field of the linear wave operator -- the scaling vector $S$, is no longer a conformal Killing vector field of the linear Klein-Gordon equation. This is essentially due to the fact that wave equations are scaling invariant but Klein-Gordon equations are not. This important difference prevents any attempt of combining na\"ivlely the techniques of wave equations with those techniques developed for Klein-Gordon equations.

On the analysis of wave-Klein-Gordon system, a first essential work is due to S. Katayama (\cite{Ka}) in $3+1$ dimensions.  Then P. LeFloch and the author have introduce the so-called hyperboloidal foliation method in the analysis of wave-Klein-Gordon system in $3+1$ dimensions (\cite{LM1}) and then applied on Einstein-scalar field system with positive mass scalar field (\cite{LM3}). This method is based on a basic observation made by L. H\"ormander in \cite{Ho1} in the analysis of Klein-Gordon equations. Then it is realized that this method can be generalized on the analysis of wave equations with some additional simple but nontrivial remarks. These techniques allow us to treat the wave equation without applying the scaling vector $S$, thus supplies a framework in which one can analyse simultaneously both wave and Klein-Gordon equations. Similar method is also applied on other type of systems such as \cite{S1}. We also recall that the application of hyperboloidal foliation supplies efficient and robust numerical method, see for example \cite{An1}, \cite{OCP}, \cite{Rinne},  \cite{Va1}, \cite{Ze1}, \cite{Ze2} etc.

The application of the hyperboloidal foliation method on wave-Klein-Gordon in $2+1$ dimension is introduced in \cite{M1} and then a system containing quasi-linear terms is treated in \cite{M2}, where a technique called the``normal form method'' is combined with the hyperboloidal foliation. This ``normal form method'', is firstly introduced by J. Shatah in \cite{Shatah85} in $3+1$ dimension, and then generalized by J-M. Delort et al. in \cite{DFX} in $2+1$ dimension combined with the vector field method on pure Klein-Gordon systems. 
 
In the present work, we continue the discussion on global existence and regard the system \eqref{eq main}. The main purpose is, during the discussion on the global existence problem, developing some new techniques, compatible with those introduced in \cite{M1} which can supply  more precise descriptions on the solution.

Roughly speaking, semi-linear Klein-Gordon terms coupled in wave equation is more difficult than both semi-linear wave terms and quasi-linear terms. Recall that in these cases we are able to get an uniform-in-time bound for some lower-order energies for the wave component, and then by Klainerman-Sobolev type inequality, we establish $s^{-1}$ decay rate of the gradient of $u$ (which is necessary to put into operation the bootstrap argument, $s^2 = t^2-r^2$). But in the present case, it is difficult to make the above strategy work anymore, because the $L^2$ norm of semi-linear terms of Klein-Gordon component (even in homogeneous linear case) has a decreasing rate as $t^{-1}$ which is no longer integrable. This makes us to establish a new estimate on $\del u$, through which we can obtain the $s^{-1}$ decay rate when the corresponding energy is not bounded.

\subsection{Basic notation and statement of the main result}
In this subsection and section \ref{sec tools} we recall some necessary notation and technical tools. For detailed description, see \cite{LM1} and \cite{M1}.

We are working in the $2+1$ dimensional space-time with standard Minkowski metric. We denote by $\Kcal := \{(t,x)\in \RR^3| t>r+1\}$ with $r = |x|$ the Euclidian norm of $x\in\RR^2$. $\Hcal_s:= \{(t,x)\in\RR^3|t = \sqrt{s^2+r^2}\}$ the hyperboloid with hyperbolic radius $s$. We denote by
$$
\Kcal_{[s_0,s_1]} := \left\{(t,x)\in\Kcal|\sqrt{s_0^2+r^2}\leq t\leq \sqrt{s_1^2+r^2}\right\}
$$
the part of $\Kcal$ limited by two hyperboloids with radii $s_0$ and $s_1$.

In the region $\Kcal$, we introduce the following vector field:
$$
L_a: = x^a\del_t + t\del_a, \quad a = 1,2
$$
which are called the {\bf Lorentz boosts}. $L_a$ together with $\del_{\alpha}$ are the vector fields which commute with the flat wave operator $\Box = \del_t\del_t - \sum_a\del_a\del_a$ and the flat Klein-Gordon operator $\Box + 1$. We use the notation $\del^IL^J$ for a product of some first-order derivatives. For example:
$$
\del_0\del_1\del_0L_1L_2 = \del^IL^J
$$
with $I=(0,1,0)$ and $J = (1,2)$ two multi-indices. We denote by $|I|$ the order (or length) of $I$.

We introduce the following $L^2$ norm on a hyperboloid: let $u$ be a function defined in $\Kcal$, then
$$
\|u\|_{L^2(\Hcal_s)}^2 := \int_{\Hcal_s}u^2dx = \int_{\RR^2}\big|u(\sqrt{s^2+r^2},x)\big|^2dx
$$

Now we are ready to state the main result.
\begin{theorem}\label{thm main}
Consider the following Cauchy problem of \eqref{eq main} with initial data posed on the initial hyperboloid $\Hcal_2$:
$$
u|_{\Hcal_2} = u_0,\quad \del_tu|_{\Hcal_2} = u_1,\quad v|_{\Hcal_2} = v_0, \quad \del_t v|_{\Hcal_2} = v_1.
$$
Then there exists a integer $N\geq 11$ and a positive constant $\vep_0$ determined by $N$ and the system \eqref{eq main} such that if for $|I|\leq N$ and $|I'|\leq N-1$,
\begin{equation}
\|\del_{\alpha}\del^{I'} u_0\|_{L^2(\Hcal_2)} + \|\del^I v_0\|_{L^2(\Hcal_2)}\leq \vep_0,\quad \|\del^{I'}u_1\|_{L^2} + \|\del^{I'}v_1\|_{L^2(\Hcal_2)}\leq \vep_0
\end{equation}
with $u_i,v_i$ supported in $\Hcal_2\cap \Kcal$ ($i = 0,1$),
then  the local-in-time solution of \eqref{eq main} associated with such initial data extends to time infinity.
\end{theorem}

\begin{remark}
The initial data is posed on a initial hyperboloid $\Hcal_2$, which seems to be not standard. But it can be seen as following: we consider the region $\Kcal_0:=\{(t,x)\in \Kcal| 2\leq t\leq \sqrt{4+r^2}\}$ and suppose that the initial data is posed on the initial plan $\{t=2\}$ with support contained in the unit disc. Then by the local theory (and finite speed of propagation), when initial data is sufficiently small, the local solution can be extended to the region $\{2\leq t\leq 5/2\}$ and supported in $\Kcal$. We see that $\{2\leq t\leq 5/2\}\supset \Kcal_0$ thus we can restrict the local solution on the hyperboloid $\Hcal_2$ and take this restriction as initial data of \eqref{eq main}.
\end{remark}

%semi-hyperboloidal frame:
%$$
%\delu_0: = \del_t,\quad \delu_a: = \frac{x^a}{t}\del_t + \del_a.
%$$

\subsection{Sketch of the proof}
Now we give a sketch of the strategy of this article.

The proof of the main theorem \ref{thm main} relays on the standard bootstrap argument. We suppose that on a (hyperbolic) time interval the local solution satisfies some energy bounds (up to certain order of derivatives), then based on these bounds we will get sharper bounds for the same energies (which do not depend on the time interval). Then by continuity argument, we conclude that the maximal interval in which the local solation satisfies those energy bounds extends to its maximal interval of existence. Then by local theory of wave system, we conclude that the local solution extends to time infinity.

In the following discussion of this subsection, for simplicity, $u$ and $v$ may refer to $\del^IL^Ju$ and $\del^IL^J v$, $\del u, \del v$ may refer to $\del^IL^J \del u, \del^IL^J\del v$.

The key point is based on the following observation. In the $2+1$ space-time:
\begin{equation}\label{eq 1 decompo-w}
\Box u =  (t-r)^{-1/2}t^{-1/2}\left((s/t)^2\del_t + 2(x^a/t)\delu_a\right)\left((t-r)^{1/2}t^{1/2}\del_tu\right) - \sum_a\delu_a\delu_au + \frac{t-r}{2t^2}\left(2+3r/t\right)\del_tu.
\end{equation}
%\marginpar{verified}
Thus
\begin{equation}\label{eq 2 decompo-w}
\mathcal{L}w + \frac{t-r}{2t^2}\left(2+(3r/t)\right)w = t^{1/2}(t-r)^{1/2}\left(\Box u + \sum_{a}\delu_a\delu_au\right)
\end{equation}
%\marginpar{verified}
where $w := t^{1/2}(t-r)^{1/2}\del_tu$, $\mathcal{L}:=(s/t)^2\del_t + 2(x^a/t)\delu_a = \left(1+(s/t)^2\right)\del_t + 2r\del_r$ (The precise definition of the symbols in the above equation will be given in the following sections). We remark that \eqref{eq 2 decompo-w} can be seen as an ODE along the integral curve of $\mathcal{L}$. What is important is that in the left-hand-side there is a positive potential $\frac{t-r}{2t^2}\left(2+(3r/t)\right)$ which leads to more decreasing rate. More precisely, let us consider the following ODE:
\begin{equation}\label{eq 1 ODE-model}
y' + A(t) y = B(t)
\end{equation}
thus
\begin{equation}\label{eq 2 ODE-model}
y(t) = y(t_0)e^{-\int_{t_0}^tA(\eta)d\eta} + \int_{t_0}^tB(\tau)e^{-\int_\tau^tA(\eta)d\eta}d\tau.
\end{equation}
We see that when $A(t)> 0$, $e^{-\int_\tau^tA(\eta)d\eta}$ becomes a decreasing factor, which relaxes the demand on the bound of $B(t)$.

Back to \eqref{eq 2 decompo-w}, on a fixed integral curve of $\mathcal{L}$, the role of $B(t)$ in \eqref{eq 1 ODE-model} is played by
$$
t^{1/2}(t-r)^{1/2}\left(\Box u + \sum_{a}\delu_a\delu_au\right)
$$
and
$$
\frac{t-r}{2t^2}\left(2+(3r/t)\right)
$$
stands for $A(t)$, which is equivalent to $(s/t)^2t^{-1}$. So in order to get uniform bound of $w = (t-r)^{1/2}t^{1/2}u \sim su$, we need in \eqref{eq 2 ODE-model}
$$
B(t)\lesssim (s/t)^2t^{-1},
$$
which means on a fixed integral curve of $\mathcal{L}$, we must have
$$
\Box u\lesssim (s/t)t^{-2},\quad \delu_a\delu_a u\lesssim (s/t)t^{-2}.
$$
which demands
\begin{equation}\label{eq 1 source-KG}
 v\left(A_1^{\alpha}\del_{\alpha}v + R v\right) + Q^{\alpha\beta}\del_{\alpha}v\del_{\beta}v\lesssim (s/t)t^{-2}
\end{equation}
and
\begin{equation}\label{eq 1 hessian-w}
\delu_a\delu_a u\lesssim (s/t)t^{-2}.
\end{equation}
But neither of them can be deduced from the bootstrap energy bounds combined with the Klainerman-Sobolev inequality. Other techniques are demanded.

\eqref{eq 1 source-KG} is based on the following observation (see also Katayama \cite{Ka} a similar case for $3+1$ dimension). Let $w:= v- P^{\alpha\beta}\del_{\alpha}u\del_{\beta}u$,
\begin{equation}\label{eq 1 KG-infty}
\Box w + w = -P^{\alpha\beta}\left(2\del_{\alpha}\Box u\cdot\del_\beta u + m^{\alpha'\beta'}\del_{\alpha'}\del_{\alpha}u\del_{\beta'}\del_{\beta}u\right)
\end{equation}
%\marginpar{finial verification till here 1:38 14-12-2017}
where the right-hand-side is linear combination of trilinear terms (because $\Box u =\text{quadratic terms}$) and $m^{\alpha'\beta'}\del_{\alpha'}\del_{\alpha}u\del_{\beta'}\del_{\beta}u$. The last one contain the hessian form of $u$, which, as explained in \cite{M1}, has better $L^{\infty}$ and energy bounds.  Then we rely on a technique developed in \cite{Kl2} (based on the hyperbolic decomposition of the Klein-Gordon operator) which leads  to \eqref{eq 1 source-KG}.

For \eqref{eq 1 hessian-w}, our strategy is to write the following identity:
$$
\delu_a\delu_a u  = t^{-2}L_aL_au -t^{-2}(x^a/t)L_au
$$
where $x^aL_au$ is not a summation. So if we arrive at the bound
\begin{equation}\label{eq 2 hessian-w}
L_a u \sim s/t,\quad  \quad L_aL_au\sim s/t
\end{equation}
then \eqref{eq 1 hessian-w} is guaranteed. But \eqref{eq 2 hessian-w} is still difficult. To get this bound, we need the following estimate:
\begin{equation}\label{eq 1 w-Kirchhoff}
\left.
\aligned
&|\Box u|\lesssim t^{-2}
\\
&u|_{\Hcal_2} ,\del_t u|_{\Hcal_2} \text{ being } C_c^{\infty}
\endaligned
\right\}
\Rightarrow |u| \lesssim s/t
\end{equation}
where the bound $|\Box u|\lesssim t^{-2}$ is guaranteed by \eqref{eq 1 source-KG}.

\eqref{eq 1 w-Kirchhoff} relays on the Kirchhoff's formula in two space dimension:
$$
\Box u = f,\quad u|_{t=2} = \del_tu|_{t=2} = 0
$$
then
$$
u(t,x) = \frac{1}{2\pi}\int_{t_0}^t\int_{|y|<t-\tau}\frac{f(\tau,x+y)}{\sqrt{(t-\tau)^2 - |y|^2}}dyd\tau.
$$
The strategy is to substitute the decay of $f$ into the above formula and get \eqref{eq 1 w-Kirchhoff}.

Now we describe the structure of this article. In section \ref{sec tools}, we recall the necessary technical tools. We will only give a sketch (for detailed explanation, see \cite{LM1} or \cite{M1}). In section \ref{sec semi-hyperbolicw}, we establish the main technical tool based on \eqref{eq 1 decompo-w}, this estimate will give the necessary decay estimate on $\del^IL^J\del u$. In section \ref{sec Kirchhoff-f}, we establish the estimate on wave component (not its gradient). This is for \eqref{eq 1 hessian-w}. In section \ref{sec other-tools} we recall the estimate on Hessian form of the wave component. Then in section \ref{sec bootstrap}, we initialize the bootstrap argument and establish some basic $L^2$ and $L^{\infty}$ bounds. In section \ref{sec hessian} we make use of the structure of the equations in order to obtain some refined $L^{\infty}$ bounds on the hessian form of the wave component. Section \ref{sec KG-infty} is devoted to the refined decay estimate on Klein-Gordon component deduced from a hyperbolic decomposition of the wave operator. Then in section \ref{sec w-infty}, we apply the observation in section \ref{sec Kirchhoff-f} and get a sufficiently sharp $L^{\infty}$ bound on the wave component. In section \ref{sec sharp-dw}, we establish the sharp decay  bound on the gradient of wave component by the techniques developed in section \ref{sec semi-hyperbolicw}. Then in the final section we conclude the bootstrap argument.

\section{Basic tools}\label{sec tools}
In this section we recall some already-established notation and results. They are explained in detail in \cite{M1}.
\subsection{The semi-hyperboloidal foliation}
In $\Kcal = \{t>|x|+1\}$, we introduce the following vector fields:
$$
\delu_0:=\del_t,\quad \delu_a:= \frac{x^a}{t}\del_t+\del_a.
$$
These vector fields forms a frame in $\Kcal$, called the semi-hyperboloidal frame. The transition matrices with respect to the natural read as:
$$
\big(\Phiu_{\alpha}^{\beta}\big)_{\beta\alpha}
=
\left(
\aligned
&1 &&0 &&&0
\\
&x^1/t &&1 &&&0
\\
&x^2/t &&0 &&&1
\endaligned
\right),
\qquad
\qquad
\big(\Psiu_{\alpha}^{\beta}\big)_{\beta\alpha}
=
\left(
\aligned
&1 &&0 &&&0
\\
-&x^1/t &&1 &&&0
\\
-&x^2/t &&0 &&&1
\endaligned
\right).
$$
with
$$
\delu_\alpha = \Phiu_{\alpha}^{\alpha'}\del_{\alpha'},
\qquad
\del_\alpha = \Psiu_{\alpha}^{\alpha'}\delu_{\alpha'}.
$$
For a two-tensor $T = T^{\alpha\beta}\del_{\alpha}\otimes\del_{\beta}$, we recall its presentation in different frames:
$$
T = T^{\alpha\beta}\del_{\alpha}\otimes\del_{\beta} = \Tu^{\alpha\beta}\delu_{\alpha}\otimes\delu_{\beta}
$$
with
$$
\Tu^{\alpha\beta} = \Psiu_{\alpha'}^{\alpha}\Psiu_{\beta'}^{\beta}T^{\alpha'\beta'},\quad
T^{\alpha\beta} = \Phiu_{\alpha'}^{\alpha}\Phiu_{\beta'}^{\beta}\Tu^{\alpha'\beta'}.
$$
Similar notation holds for three-tensor:
$$
P = P^{\alpha\beta\gamma}\del_{\alpha}\otimes\del_{\beta}\otimes\del_{\gamma} =  \Pu^{\alpha\beta\gamma}\delu_{\alpha}\otimes\delu_{\beta}\otimes\delu_{\gamma}
$$
and we also have similar relation of transition.
$$
\Pu^{\alpha\beta\gamma} = \Psiu_{\alpha'}^{\alpha}\Psiu_{\beta'}^{\beta}\Psiu_{\gamma'}^{\gamma}P^{\alpha'\beta'\gamma'},\quad
P^{\alpha\beta\gamma} = \Phiu_{\alpha'}^{\alpha}\Phiu_{\beta'}^{\beta}\Phiu_{\gamma'}^{\gamma}\Pu^{\alpha'\beta'\gamma'}.
$$
The functions $\Psiu_{\alpha}^{\beta}$ are smooth homogeneous functions defined in $\Kcal$ (see in detail in the following subsections). The following bounds in $\Kcal$ is guaranteed by lemma \ref{lem 1 homo}:
\begin{equation}\label{eq 1 homogeneous}
|\del^IL^J \Psiu_{\alpha}^{\beta}| + \leq C(N)t^{-|I|},\quad |I|+|J|\leq N.
\end{equation}
So we observe that
$$
|\del^IL^J\Tu^{\alpha\beta}| + |\del^IL^J\Pu^{\alpha\beta\gamma}|\leq C(N)t^{-|I|}.
$$

%Furthermore, if $N$ is a null from, i.e, \eqref{eq 1 null} holds.  Then (see \cite{M1}, lemma 5.1):
%\begin{equation}\label{eq 2 null}
%|\del^IL^J\Pu^{00}| \leq C(N)(s/t)^2,\quad \forall \,|I|+|J|\leq N.
%\end{equation}

\subsection{The energy estimate}
We recall the energy estimate on hyperboloids. Recall that
$$
\aligned
E_{c^2}(s,u):=&\int_{\Hcal_s}\big(|\del_tu|^2+\sum_a|\del_au|^2 + 2(x^a/t)\del_tu\del_au + c^2u^2 \big)dx
\\
=&\int_{\Hcal_s}\big(\sum_a |\delu_a|^2 + |(s/t)\del_tu|^2 + c^2u^2\big)dx
\\
=&\int_{\Hcal_s}\big(|\delu_{\perp}u|^2 + \sum_a|(s/t)\del_a u|^2 + \sum_{a<b}\big|t^{-1}\Omega_{ab}u\big|^2 + c^2u^2\big) dx.
\endaligned
$$
with
$$
\Omega_{ab} := x^a\del_b-x^b\del_a,\quad \delu_{\perp} := \del_t + \frac{x^a}{t}\del_a.
$$
Then we state the following energy estimate, which is a special case of proposition 7.1 of \cite{M1}.
\begin{proposition}[Energy estimate]\label{prop energy}
Let $u$ be a $C^2$ function defined in $\Kcal_{[s_0,s_1]}$ and vanishes near the conical boundary $\del\Kcal := \{t=r+1\}$. Suppose that
$$
\Box u + c^2u = f.
$$
Then for $s_0\leq s\leq s_1$, the following estimate holds:
\begin{equation}\label{eq 1 energy}
E_{c^2}(s,u)^{1/2}\leq E_{c^2}(s_0,u)^{1/2} + \int_{s_0}^s \|f\|_{L^2(\Hcal_\tau)}d\tau
\end{equation}
\end{proposition}
\begin{proof}[Sketch of proof]
We remark that
$$
\del_tu\cdot \Box u = \frac{1}{2}\del_t\left(|\del_t u|^2 + \sum_a|\del_au|^2 + c^2u^2\right) - \sum_a\del_a\left(\del_tu\del_au\right)
$$
Integrate this in the region $\Kcal_{[s_0,s]}$ and apply the Stokes formula. We remark that the unit normal vector field of $\Hcal_s$ (with respect to the Euclidian metric) is
$$
\left(\frac{1}{\sqrt{t^2+r^2}},\frac{-x^a}{\sqrt{t^2+r^2}}\right),
$$
and the volume element of $\Hcal_s$ is (with respect to the Euclidian norm) is
$$
d\sigma = \frac{\sqrt{t^2+r^2}}{t}dx,
$$
then
$$
E_{c^2}(s,u) - E_{c^2}(s_0,u) = 2\int_{\Kcal_{[s_0,s]}}\del_t u\cdot f dxdt = 2\int_{s_0}^s\int_{\Hcal_\tau}(s/t)\del_tu\cdot f dxd\tau
$$
Then we see that (derive the above identity with respect to $s$)
$$
\aligned
\frac{d}{ds}E_{c^2}(s,u) =& 2\int_{\Hcal_\tau}(s/t)\del_tu\cdot f dxd\leq 2\|(s/t)\del_tu\|_{L^2(\Hcal_s)}\cdot\|f\|_{L^2(\Hcal_s)}
\\
\leq& 2E_{c^2}(s,u)^{1/2}\cdot\|f\|_{L^2(\Hcal_s)}
\endaligned
$$
which leads to
$$
\frac{d}{ds}E_{c^2}(s,u)^{1/2}\leq \|f\|_{L^2(\Hcal_s)}
$$
which leads to the desired result.
\end{proof}

\subsection{The Klainerman-Sobolev inequalities}
To turn the $L^2$ bounds into $L^{\infty}$ bounds, we need the following Klainerman-Sobolev type inequality on hyperboloids. For the proof, see for example \cite{Ho1} or \cite{LM1}.
\begin{proposition}
Let $u$ be a $C^2$ function defined in $\Kcal[s_0,s_1]$ and vanishes near the conical boundary $\del\Kcal$. Then for $s_0\leq s\leq s_1$, the following inequality holds:
\begin{equation}\label{ineq 1 KS}
t|u(t,x)|\leq C\sum_{|I|+|J|\leq 2}\|\del^IL^Ju\|_{L^2(\Hcal_s)}
\end{equation}
where $C>0$ is a universal constant.
\end{proposition}

\subsection{Homogeneous functions}
We recall the notion of homogeneous functions. Let $u$ be a function defined in $\RR^4$, sufficiently regular. $u$ is called to be homogeneous of degree $k\in\mathbb{Z}$, if the following two conditions hold:
\\
$\bullet$ $\del^I u(1,x)$ is bounded on the disc $|x|\leq 1$.
\\
$\bullet$ $u(\lambda t, \lambda x) = \lambda^ku(t,x), \forall \lambda>0$.

The following properties are direct and we omit the proof:
\begin{lemma}\label{lem 1 homo}
Suppose that $u,v$ are homogeneous of degree $k,l$ respectively. Then
\\
$-$ when $k=l$, $\alpha u+\beta v$ is homogeneous of degree $k$, where $\alpha,\beta$ are constants
\\
$-$ $uv$ is homogeneous of degree $k+l$,
\\
$-$ $\del^IL^J u$ is homogeneous of degree $k-|I|$,
\\
$-$ $|u|\leq Ct^k$.
\end{lemma}

We remark that $\Psiu_{\alpha}^{\beta}$ are homogeneous of degree zero, and $t^{-1}$ is homogeneous of degree $-1$.

Furthermore, we pay special attention on the function $s/t$.
\begin{lemma}\label{lem 2 homo}
Let $(I,J)$ be a pair of multi-indices. Then
$$
L^J(s/t) = \Lambda^J(s/t), \quad \del^I(s/t) = \sum_{I_1+I_2+\cdots+I_k=I\atop |I_j|\geq 1}\Lambda_k^{I}(s/t)^{1-2k},
$$
where $\Lambda^J$ is homogeneous of degree zero, $\Lambda^{I}_k$ are homogeneous of degree $-|I|$.

Furthermore, in $\Kcal$,
\begin{equation}\label{eq 1 lem 2 homo}
|\del^IL^J(s/t)|\leq \left\{
\aligned
&C(s/t),\quad |I|=0,
\\
&Cs^{-1},\quad |I|>0,
\endaligned
\right.
\end{equation}
where $C$ is a constant determined by $I,J$.
\end{lemma}
\begin{proof}
We remark that $L_a(s/t) = -\frac{x^a}{t}(s/t)$ and $\frac{-x^a}{t}$ is homogeneous of degree zero. Then by induction the $L^J(s/t)$ can be easily calculated.

For $\del^I(s/t)$, we remark that
$$
s/t = (1-(r/t)^2)^{1/2}.
$$
We denote by $u = (s/t)^2 = 1-(r/t)^2$ and
$$
\aligned
f:\RR^+ &\longrightarrow \RR^+
\\
x&\longrightarrow x^{1/2}
\endaligned
$$
so $(s/t) = f(u)$ with $u$ a homogeneous function of degree zero. Then by Fa\`a di Bruno's identity, we remark that
$$
\del^I(f(u)) = \sum_{I_1+I_2+ \cdots +I_k=I}f^{(k)}(u)\cdot \del^{I_1}u\del^{I_2}u\cdots \del^{I_k}u,\quad |I_j|\geq 1, j=1,2,\cdots k.
$$
Recall that
$$
f^{(k)}(u) = C(k)u^{-k+1/2} = C(k)(s/t)^{1-2k}.
$$
Furthermore, we remark that $\del^{I_1}u\del^{I_2}u\cdots \del^{I_k}u$ is homogeneous of degree $-|I|$. So the decomposition on $\del^I(s/t)$ is established.

For the estimate, we observe that when $|I| =0$ it is direct. When $|J|=0$, we remark that
$$
|f^{(k)}(u)\cdot \del^{I_1}u\del^{I_2}u\cdots \del^{I_k}u|\leq C(k)(t/s)^{2k-1}t^{-|I|}.
$$
Recall that $k\leq |I|$ and $t/s\geq 1$, and the fact that $t\leq s^2$ in $\Kcal$, then
$$
|f^{(k)}(u)\cdot \del^{I_1}u\del^{I_2}u\cdots \del^{I_k}u|\leq C(|I|)t^{|I|-1}s^{-2|I|+1}\leq C(|I|)s^{-1}
$$
which leads to
\begin{equation}\label{eq 1 proof lem 2 homo}
|\del^I(s/t)|\leq C(|I|)s^{-1}.
\end{equation}

When $|I|>0$,
$$
\del^IL^J(s/t) = \del^I\left(\Lambda^J(s/t)\right) = \sum_{I_1+I_2=I}\del^{I_1}\Lambda^J\cdot \del^{I_2}(s/t).
$$
Now, remark that $\del^{I_1}\Lambda^J$ is homogeneous of degree $-|I_1|\leq 0$, thus bounded by $Ct^{-|I_1|}$. Then by \eqref{eq 1 proof lem 2 homo}, we remark that the desired result is established.
\end{proof}

\subsection{The commutators}\label{sec comm}
%\marginpar{: This section is to be re-verified}
In this section we will classify the high-order derivatives composed by $\{\del_{\alpha},\,\delu_\alpha,\,L_a\}$. We will list out some key decompositions of commutators established in \cite{LM1} chapter 3 and the give some estimates based on these decompositions.

\begin{lemma}\label{lem 1 decompo commu}
Let $u$ be a function defined in $\Kcal$, sufficiently regular. Let $(I,J)$ be a pair of multi-indices, then the following relations hold:
\begin{equation}\label{eq 1 comm}
[\del_{\alpha},L^J] = \sum_{\beta,|J'|<|J|}\Gamma_{\alpha J'}^{J\,\beta}\del_{\beta}L^{J'},
\end{equation}
\begin{equation}\label{eq 2 comm}
[\del^I,L^J] = \sum_{|I'|=|I|\atop|J'|<|J|}\Gamma_{I'J'}^{IL}\del^{I'}L^{J'}
\end{equation}
where $\Gamma_{\alpha J'}^{J\,\beta}$ and $\Gamma_{I'J'}^{IL}$ are constants.

%\begin{equation}\label{eq 3 comm}
%[\delu_a,L^J] = \sum_{b,|J'|<|J|}\Gammau_{aJ'}^{Jb} \delu_bL^{J'},
%\end{equation}
%\begin{equation}\label{eq 4 comm}
%[\delu_a,\del^I] = \sum_{|I'|=|I|-k,k\geq 0\atop }\Deltau_{I'}^{I}\del^{I'}L_a - t^{-1}\sum_{|I'|=|I|}\Gamma_{aI'}^{I}\del^{I'}
%\end{equation}
%where $\Deltau_{I'}^I$ are homogeneous of degree $(-k-1)$ and $\Gamma_{aI'}^I$ are constants.
\end{lemma}

Now we introduce the following notation: let $Z^I u$ be an $N-$order derivative acting on $u$ with the order of the multi-index $|I|=N$ and $Z$ represents one of the following derivative: $\del_{\alpha}$ (partial derivatives), $\delu_a$ (tangent derivatives) and $L_a$ (boosts). Suppose that $I = (I_1,I_2,\cdots I_N)$, $I_j\in\{0,1,2,3,4,5,6,7,8,9\}$ with
$$
Z_{I_j} = \left\{
\aligned
&\del_{I_j}, \quad I_j = 0,1,2,3
\\
&L_{I_j-3},\quad I_j = 4,5,6,
\\
&\delu_{I_j-6},\quad I_j = 7,8,9.
\endaligned
\right.
$$
Then we are going to classify the $N-$order operators $\{Z^I\}$ into four types:
\begin{itemize}
\item {\bf 0.} in $Z^I$ all the derivatives are $L_a$, which is equivalent to demand $\forall 4\leq I_j\leq 6$.
\\
\item {\bf 1.} in $Z^I$ all derivatives are $\del_{\alpha}$ and $L_a$ ($I_j\leq 6$). If it contains $(1+d)$ partial derivatives and $l$ boosts $L_a$ ($d\geq 0, l+d+1=|I|$), then by \eqref{eq 1 comm} and \eqref{eq 2 comm}
\begin{equation}\label{ineq 1.1 comm}
|Z^Iu|\leq C\sum_{\alpha,|I'|=d\atop |J'|\leq l}|\del_{\alpha}\del^{I'}L^{J'}u|.
\end{equation}
\\
\item {\bf 2.} in $Z^I$ there are $(1+b)$ tangent derivatives $\delu_a$, $l$ boosts and $d$ partial derivatives with $l,d,b\geq 0$. In this case we make the following observation:
    $$
    \delu_a u = t^{-1}L_au
    $$
    and $t^{-1}$ is homogeneous of degree $-1$. Now suppose that $I = (I_1,I_2,\cdots I_N)$, and suppose that $I_{k_1}$ is the first element which is strictly large than $6$ (or equivalently saying, $Z_{I_{k_1}}$ is the first tangent derivative). Then we can write
    $$
    Z^I = Z^{J}\delu_aZ^{K}, \quad J = (I_1,\cdots I_{k_1-1}), \quad K = (I_{k_1+1}\cdots I_N), \quad a = I_{k_1} - 6.
    $$
    Remark that $Z^J$ belongs to the type {\bf 1}. Then
    $$
    Z^Iu = Z^J(t^{-1}L_aZ^Ku) = \sum_{J_1+J_2=J}Z^{J_1}t^{-1}\cdot Z^{J_2}L_aZ^Ku
    $$
    where $Z^{J_1}$ and $Z^{J_2}$ are of type {\bf 1}. Remark that $Z^{J_1}t^{-1}$ is homogeneous of degree $\leq -1$. Thus we see
    $$
    \big|Z^{J_1}t^{-1}\cdot Z^{J_2}L_aZ^Ku\big|\leq Ct^{-1}\big|Z^{J_2}L_aZ^Ku\big|.
    $$
    If, in $Z^{J_2}L_aZ^K$ there is no tangent derivative $\delu_a$, we conclude that
    $$
    |Z^Iu|\leq Ct^{-1}\sum_{|K'|\leq |K|}\big|Z^Ku\big|
    $$
    with $Z^K$ belong to type {\bf 1}. If in $Z^{J_2}L_aZ^K$ still contains $\delu_a$ (of course it contains strictly less tangent derivatives than $Z^I$ does), we take this operator and repeat the above argument on $Z^I$. Because $Z^I$ contains finite many $\delu_a$, this procedure will terminate in finite many steps. Thus we conclude :
    \begin{equation}\label{ineq 2.1 comm}
    |Z^Iu|\leq C(|I|)t^{-k}\sum_{|I'|+|J'|\leq|I|}|\del^{I'}L^{J'}u|
    \end{equation}
    with $k$ the number of tangent derivatives contained in $Z^I$. When $k=1$, the above estimate can be written into the following form:
    \begin{equation}\label{ineq 2.2 comm}
    |Z^Iu|\leq C(|I|)\sum_{b \atop |I'|+|J'|\leq |I|-1}\big|\delu_b\del^{I'}L^{J'}u\big| + C(|I|)s^{-1}\sum_{|I'|\leq |I|}\big|(s/t)\del^{I'}u\big|.
    \end{equation}
%\marginpar{Verified}
%    $$
%    |Z^Iu|\leq C\sum_{|I'|+|J'|\leq n-1}|\delu_a\del^{I'}L^{J'}u|,
%    $$
%    $$
%    t^{-1}|\del_{\alpha}\del^IL^Ju|\leq s^{-1}|(s/t)\del_{\alpha}\del^IL^J u|.
%    $$
%    Thus in this case we see
%    \begin{equation}\label{ineq 2.1 comm}
%    |Z^Iu|\leq C\!\!\!\!\sum_{|I'|+|J'|\leq n-1}|\delu_a\del^{I'}L^{J'}u| +C s^{-1}\!\!\!\!\!\!\!\sum_{|I'|+|J'|\leq n-1}|(s/t)\del_{\alpha}\del^{I'}L^{J'}u|.
%    \end{equation}
%    Further more,
%    \begin{equation}\label{ineq 2.2 comm}
%    |Z^Iu|\leq Ct^{-1}\!\!\!\!\sum_{|I'|+|J'|\leq n}|\del^{I'}L^{J'}u| + Ct^{-1}\!\!\!\!\sum_{|I'|+|J'|\leq n}|\del_{\alpha}\del^{I'}L^{J'}u|.
%    \end{equation}
%\\
%\item {\bf 3.} in $Z^I$ there are at least two $\del_{\alpha}$. By \eqref{}
%\item {\bf 3.} in $Z^I$ there are at least one $\delu_a$ and one $\del_{\alpha}$, or one $\delu_a$ and one $\delu_{\alpha}$. In this case the following estimates hold:
%    $$
%    |Z^Iu|\leq Ct^{-1}\sum_{\alpha,|I'|+|J'|\leq N-2}|\del^{I'}L^{J'}\del_{\alpha}u|
%    $$
%    and
%    $$
%    |Z^Iu|\leq Ct^{-1}\sum_{\alpha,|I'|+|J'|\leq N-2}|\del^{I'}L^{J'}\del_{\alpha}u|
%    $$
\end{itemize}

%\subsection{Homogeneous coefficients}
%We recall the notion of homogeneous coefficients introduced in \cite{LM1}. These are the $C^{\infty}$ functions defined in $\{t>r\}$, bounded on $\Hcal_2\cap\Kcal$ and satisfies the following condition:
%$$
%f(\lambda t,\lambda x) = \lambda^{\alpha}f(t,x),\quad \forall (t,x)\in \Kcal, \quad \lambda>0.
%$$
%Then $f$ is called a homogeneous coefficient of degree $0$.
%Then the following bounds hold:
%\begin{equation}\label{eq 1 homo}
%|\del^IL^J f(t,x)|\leq C(N)t^{-|I|+\alpha}.
%\end{equation}
%We remark that the coefficients of $\Phiu$ and $\Psiu$ are homogeneous of degree $0$, because these are of form $x^a/t$.

\section{The semi-hyperboloidal decomposition of the wave operator}\label{sec semi-hyperbolicw}

\subsection{The semi-hyperboloidal decomposition}
We recall the semi-hyperboloidal frame defined in $\Kcal$:
$$
\delu_0 := \del_t,\quad \delu_a := \frac{x^a}{t}\del_t + \del_a.
$$
Then we write
$$
\Box u = \left((s/t)^2\del_t + 2(x^a/t)\delu_a\right)\del_tu - \sum_a\delu_a\delu_au + t^{-1}\left(2-(r/t)^2\right)\del_tu
$$
which leads to
\begin{equation}\label{eq 3 decompo-w}
\Box u =  t^{-1/2}\left((s/t)^2\del_t + 2(x^a/t)\delu_a\right)\left(t^{1/2}\del_t\right) - \sum_a\delu_a\delu_a + \frac{3}{2t}(s/t)^2\del_t.
\end{equation}
%\marginpar{verified}
Then we see \eqref{eq 1 decompo-w}:
$$
\aligned
\Box u =&  (t-r)^{-1/2}t^{-1/2}\left((s/t)^2\del_t + 2(x^a/t)\delu_a\right)\left((t-r)^{1/2}t^{1/2}\del_tu\right) + \frac{t-r}{2t^2}\left(2+3r/t\right)\del_tu
\\
 &- \sum_a\delu_a\delu_au.
\endaligned
$$
In the above equations $u$ is supposed to be sufficiently regular. Now as we explained in section \ref{sec intro}, we denote by
$$
w:= (t-r)^{1/2}t^{1/2}\del_t u
$$
and
$$
\mathcal{L} := (s/t)^2\del_t + 2(x^a/t)\delu_a = \left(1+(r/t)^2\right)\del_t + 2(r/t)\del_r.
$$
Then \eqref{eq 1 decompo-w} can be written into the following form:
\begin{equation}\label{eq 5 decompo-w}
\mathcal{L} w + \frac{t-r}{t^2}\left(2+3r/t\right)w = (t-r)^{1/2}t^{1/2}\left(\Box u + \sum_a\delu_a\delu_au\right).
\end{equation}
%\marginpar{verified}
In the following subsection we will establish the $L^{\infty}$ estimate on $\del_t u$ based on the above identity.

\subsection{The $L^{\infty}-L^{\infty}$ estimate on $\del u$}
We write \eqref{eq 5 decompo-w} into the following form:
\begin{equation}\label{eq 1 Linfty-dw}
L w + \frac{t-r}{2t^2}\cdot\frac{2+(3r/t)}{1+(r/t)^2}w = \left(1+(r/t)^2\right)^{-1}(t-r)^{1/2}t^{1/2}\left(\Box u + \sum_a\delu_a\delu_au\right).
\end{equation}
where
$$
L:=\del_t + \frac{2rt}{t^2+r^2}\del_r
$$
Then we denote by $\gamma(t;t_0,x_0)$ the integral curve of $L$ passing $(t_0,x_0)$ when $t = t_0$. We define $w_{t_0,x_0}(t): = w(\gamma(t;t_0,x_0))$ which is the restriction of $w$ on $\gamma(t;t_0,x_0)$. Then we see that \eqref{eq 1 Linfty-dw} is written as
\begin{equation}\label{eq 2 linfty-dw}
w_{t_0,x_0}' + P(t;t_0,x_0)w_{t_0,x_0}
= \left(1+(r/t)^2\right)^{-1}(t-r)^{1/2}t^{1/2}\left(\Box u + \sum_a\delu_a\delu_au\right)\bigg|_{\gamma(t;t_0,x_0)}
\end{equation}
with
$$
P(t;t_0,x_0):= \frac{t-r}{2t^2}\cdot\frac{2+(3r/t)}{1+(r/t)^2}\bigg|_{\gamma(t;t_0,x_0)}\geq \frac{t-r}{t^2}\bigg|_{\gamma(t;t_0,x_0)}.
$$
%\marginpar{Verified}
To estimate $w_{t_0,x_0}$, our strategy is to integrate the above ODE along $\gamma$ from the point where $\gamma$ meets the boundary of $\Kcal_{[2,s_0]}$ (where $s_0 = \sqrt{t_0^2-r_0^2}$) to $t_0$. So we need to guarantee that if we start from $(t_0,x_0)$ and go backward along $\gamma(t,t_0,x_0)$, then we must meet the conical boundary $\del \Kcal = \{(t,x)|t = r+1\}$ or the initial slice $\Hcal_2$; furthermore, the arc limited by these two points is contained in the region $\Kcal_{[2,s_0]}$.  This is concluded in the following lemma:
\begin{lemma}\label{lem 1 decompo-w}
Let $L$ be the vector field defined in the region $\Kcal$:
$$
L:= \del_t + \frac{2rt}{t^2+r^2}\del_r.
$$
Suppose $(t_0,x_0)\in \Kcal_{[2,+\infty)}$ and denote by $\gamma (t;t_0,x_0)$ the integral curve passing $(t_0,x_0)$ with $\gamma(t_0;t_0,x_0) = (t_0,x_0)$. Then for each point $(t_0,x_0)\in \Kcal_{[2,+\infty)}$, there exists a $\tau_0$ such that
$$
2\leq \tau_0<t_0,\quad \gamma(\tau_0;t_0,,x_0)\in \{(t,x)|t=r+1\}\cup\Hcal_2,
$$
and $\forall \tau_0\leq t\leq t_0$
$$
\gamma(t;t_0,x_0)\in \Kcal_{[2,s_0]}, \quad s_0 = \sqrt{t_0^2-r_0^2}.
$$
\end{lemma}
\begin{proof}
We first remark that in the region $\Kcal_{[2,+\infty)}$, along the curve $\gamma(t;t_0,x_0)$, $s = \sqrt{t^2-r^2}$ is strictly increasing:
$$
L s^2 = L(t^2-r^2) = 2t - \frac{4r^2t}{t^2+r^2} =   \frac{2t}{t^2+r^2}\cdot (t^2-r^2) = \frac{2ts^2}{t^2+r^2}
$$
thus
$$
L(\ln s^2) = \frac{2t}{t^2+r^2}>0
$$
thus $s$ is strictly increasing along $\gamma(t;t_0,x_0)$.

Thus, for $(t_0,x_0)\in \Kcal_{[2,+\infty)}$, there exists a $2\leq t_0'<t_0$ such that $\forall t_0'< t < t_0$, $\gamma(t;t_0,x_0)\in \Kcal_{[2,s_0)}$. We take
$$
\tau_0 := \inf\{t_0'|\forall\, t_0'<t<t_0,\, \gamma(t;t_0,x_0)\in\Kcal_{[2,s_0)}\}.
$$
We will prove that $\gamma(\tau_0;t_0,x_0)$ is contained in $\del\Kcal\cup \Hcal_2$. First we prove that $\gamma(\tau_0;t_0,x_0)\in \del\Kcal_{[2,s_0]}$. This is because if not so, by the continuity (the interior of $\Kcal_{[2,s_0]}$), there exists a $t_0''<t_0'$ such that $\gamma(t_0'';t_0,x_0)\in \Kcal_{[2,s_0)}$ which is a contradiction.

Then we remark that $\gamma(\tau_0;t_0,x_0)\notin \Hcal_{s_0}$, that is  because the hyperbolic distance form $\gamma(t;t_0,x_0)$ to the origin is strictly increasing, and we remark that $\tau_0<t_0$.
\end{proof}
%\marginpar{Verified}

Based on the above lemma, we are ready to establish the following estimate:
\begin{proposition}\label{prop 1 dw}
Let $u$ be a $C^2$ function defined in $\Kcal_{[2,s_0]}$, and vanishes near the conical boundary $\{t = r+1\}$. Then for any point $(t,x)\in \Kcal_{[2,s_0]}$, the following estimate holds:
\begin{equation}\label{eq 1 prop-main Linfty-dw}
|\del_tu(t,x)|\leq Cs^{-1}\|\del_t u\|_{L^{\infty}(\Hcal_2)} + s^{-1}\int_2^te^{-\int_\tau^tP(\eta;t_0,x_0)d\eta}|R_w(\tau;t,x)|d\tau
\end{equation}
\end{proposition}
where
$$
R_w(\tau;t,x) := (1+(r/t)^2)^{-1}(t-r)^{1/2}t^{1/2}\left(\sum_{a}\delu_a\delu_a u + \Box u\right)\bigg|_{\gamma(\tau;t,x)}.
$$
\begin{proof}
With the above notation, we write \eqref{eq 2 linfty-dw} into the following form:
$$
w_{t,x}'(\tau) + P(\tau;t,x)w_{t,x}(\tau) = R_w(\tau;t,x) .
$$
Then by standard ODE argument, we see that
$$
w_{t,x}(t) = w_{t,x}(\tau_0)e^{-\int_{\tau_0}^tP(\eta;t,x)d\eta} + \int_{\tau_0}^te^{-\int_\tau^tP(\eta;t,x)d\eta}R_w(\tau;t,x)d\tau
$$
with
$$
P(t;t_0,x_0) := \frac{t-r}{2t^2}\cdot\frac{2+(3r/t)}{1+(r/t)^2}\geq \frac{t-r}{t^2},
$$
thus
$$
|w_{t,x}(t)|\leq |w_{t,x}(\tau_0)| + \int_2^te^{-\int_\tau^tP(\eta;t,x)d\eta}|R_w(\tau;t,x)|d\tau
$$
which leads to the desired result.
\end{proof}

\section{The $L^{\infty}-L^{\infty}$ estimate based on Kirchhoff's formula}\label{sec Kirchhoff-f}
As already explained in introduction, we need to establish a ``good'' bound of the term $R_w$ in \eqref{eq 1 prop-main Linfty-dw} in order to get sufficient bound on $\del_t u$. To do so, we need the following estimate:
\begin{proposition}\label{prop 1 w-Kirchhoff}
Let $u$ be a $C^2$ function defined in $\Kcal_{[2,s_0]}$, vanishes near the conical boundary $\del\Kcal := \{t=r+1\}$ and satisfies the following Cauchy problem:
\begin{equation}\label{eq 1 prop 1 w-Kirchhoff}
\left\{
\aligned
&\Box u = F,
\\
&u|_{\Hcal_2} = u_0,\quad \del_t u|_{\Hcal_2} = u_1
\endaligned
\right.
\end{equation}
with $u_0,u_1$ being $C_c^{\infty}$ and supported in $\Hcal_2\cap \Kcal$. Suppose that $F$ is defined in $\Kcal_{[2,s_0]}$ and vanishes near $\Kcal$, and satisfies the following bound:
\begin{equation}\label{eq 2 prop 1 w-Kirchhoff}
|F(t,x)|\leq C_Ft^{-2}.
\end{equation}
Then the following estimate holds:
\begin{equation}\label{eq 3 prop 1 w-Kirchhoff}
|u(t,x)|\leq CC_F(s/t) + C_0 s^{-1}
\end{equation}
with $C_0$ a constant determined by $u_0$ and $u_1$.
\end{proposition}
%\marginpar{Finial verification till here 19:43 14-12-2017}
Remark that $s^{-1}\leq (s/t)$ in $\Kcal$. The proof of this result is as following: first we decompose $u$ in the following manner:
\begin{equation}
\Box w_1 = F,\quad w_1|_{\Hcal_2} =  w_1|_{\Hcal_2} = 0
\end{equation}
and
$$
\Box w_2 = 0,\quad w_2|_{\Hcal_2} = u_0,\quad \del_t w_2|_{\Hcal_2} = u_1.
$$
It is clear (by Kirchhoff's formula for homogeneous equations) that
$$
|w_2(t,x)|\leq C_0s^{-1}
$$
with $C_0$ a constant determined by $u_0$ and $u_1$. The main difficulty comes from the bound on $w_1$. We recall
$$
w_1(t,x) = \frac{1}{2\pi}\int_{t_0}^t\int_{|y|<t-\tau}\frac{F(\tau,x+y)}{\sqrt{(t-\tau)^2 - |y|^2}}dyd\tau
$$
The strategy is to substitute the bound of $F$ into the above expression and get the desired bound. This is long and tedious calculation and we decompose this task in serval parts.

First, we remark that
$$
\aligned
|w_1(t,x)|\leq& \frac{C_F}{2\pi}\int_{t_0}^t\int_{|y|\leq t-\tau}\frac{\mathds{1}_{\{|x+y|<\tau-1\}}dyd\tau }{\tau^2\sqrt{(t-\tau)^2 - |y|^2}}
\\
=&\int_{t_0/t}^1 \lambda^{-2}\int_{|\omega|\leq 1-\lambda}\frac{\mathds{1}_{\{|x/t + \omega|\leq \lambda-t^{-1}\}}}{\sqrt{(1-\lambda)^2-|\omega|^2}}d\omega d\lambda
\endaligned
$$
where
$$
\lambda := \tau/t,\quad \omega := y/t.
$$
We define
$$
I(\lambda) := \int_{|y|\leq 1-\lambda}\frac{\mathds{1}_{\{|x/t + y|\leq \lambda-t^{-1}\}}}{\sqrt{(1-\lambda)^2-|y|^2}}dy
$$
and the desired result is equivalent to the following inequality:
\begin{equation}\label{eq 1 pr prop-main Linfty-w}
\int_{t_0/t}^1\lambda^{-2}I(\lambda)d\lambda\leq CC_F(s/t).
\end{equation}
And this inequality will be guaranteed by the following two lemmas \ref{lem 1 w-Kirchhoff}, \ref{lem 2 w-Kirchhoff}.

We first give some general description: $I(\lambda)$ is the integration of the function $((1-\lambda)^2-|y|^2)^{-1/2}$ on the region
$$
D_{\lambda} = \{y\in\mathbb{R}^3||y + x/t|\leq \lambda-t^{-1},|y|\leq 1-\lambda\}.
$$

Without lose of generality, we make the following convention: let $x = (r,0)$. For each $t_0/t\leq \lambda\leq 1$, we denote by
$$
D_0(\lambda) = \{y\in\mathbb{R}^2\big||y|\leq 1-\lambda\}, \quad D_1(\lambda) = \{y\in\mathbb{R}^2\big||y+x/t|\leq \lambda-t^{-1}\}
$$
and we see that $D_\lambda = D_0(\lambda)\cap D_1(\lambda)$. We denote by $X = (-r/t,0)$ and $O = (0,0)$. Then on $\mathbb{R}^2$ we introduce the following parametrization:
$$
\rho:=|y|,\quad \theta:= \text{the angle from $(1,0)$ to $y$}
$$
Then
$$
\aligned
I(\lambda) = \int_{\rho_0}^{\rho_1}\frac{\rho d\rho}{\sqrt{(1-\lambda)^2-\rho^2}}\int_{\theta(\rho)}^{2\pi-\theta(\rho)}d\theta
=2\int_{\rho_0}^{\rho_1}\frac{(\pi-\theta(\rho))\rho d\rho}{\sqrt{(1-\lambda)^2-\rho^2}}.
\endaligned
$$
The bounds of integration is determined by the relative position of the two discs. More precisely by the following two criteria:
\begin{itemize}
\item  Is one disc contained in the other? There are three cases: {\bf I}, $D_0(\lambda)\supset D_1(\lambda)$; {\bf II} $D_0(\lambda)\cap D_1(\lambda)$ is non-empty but neither is contained in the other; {\bf III} $D_0(\lambda)\subset D_1(\lambda)$.
\\
 \item Is the center of $D_0(\lambda)$ contained in $D_1(\lambda)$? There are  two cases: {\bf A}, $O\notin D_1(\lambda)$; {\bf B}, $O\in D_1(\lambda)$.
\end{itemize}
Then we use a composed index such as ${\bf IIA}$ to describe the relative position of the two discs. For $\lambda$ varies from $t_0/t$ to $1$, the relation between the two discs can be showed by the following table:
$$
{\bf IA} \quad \frac{t-r+1}{2t}\quad {\bf IIA}\quad \frac{r+1}{t}\quad {\bf IIB}\quad \frac{t+r+1}{2t}\quad{\bf IIIB}\quad \text{for } r>\frac{t-1}{3}
$$
$$
{\bf IA} \quad \frac{r+1}{t} \quad  {\bf IB}\quad \frac{t-r+1}{2t} \quad {\bf IIB}\quad \frac{t+r+1}{2t}\quad {\bf IIIB}\quad \text{for } r<\frac{t-1}{3}
$$
The value between two cases such as $\frac{t-r+1}{2t}$ between {\bf IA} and {\bf IIA} means that $\lambda = \frac{t-r+1}{2t}$ separates this two cases: when $\lambda < \frac{t-r+1}{2t}$ we are in {\bf IA} and when $\lambda>\frac{t-r+1}{2t}$ we are in {\bf IIA}.
%We thus have the following cases:
%$$
%{\bf IA,\quad IB,\quad IIA,\quad IIB,\quad IIIB.}
%$$

Now we state the two lemmas:

\begin{lemma}\label{lem 1 w-Kirchhoff}
Taking the above notation. Let $\epsilon_0>0$, then let $(t,x)\in \Kcal$ with $\frac{t-r}{t}>\epsilon_0$, then
\begin{equation}\label{eq 1 lem 1 w-Kirchhoff}
\int_{t_0/t}^1\lambda^{-2}I(\lambda)d\lambda\leq C(\epsilon_0)C_F.
\end{equation}
\end{lemma}

\begin{lemma}\label{lem 2 w-Kirchhoff}
There exists a positive constant $\delta_0>0$ such that for a point $(t,x)\in\Kcal$ with $0<\frac{t-r}{t}<\delta_0$,
\begin{equation}\label{eq 1 lem 2 w-Kirchhoff}
\int_{t_0/t}^1\lambda^{-2}I(\lambda)d\lambda\leq CC_F(s/t).
\end{equation}
\end{lemma}

\begin{proof}[Proof of lemma \ref{lem 1 w-Kirchhoff}]
We remark that for any $\epsilon_0>0$, if we consider the region $\Kcal_{\epsilon_0} := \{(t,x)\in\Kcal|(t-r)/t>\epsilon_0\}$, then in $\Kcal_{\epsilon_0}$, $(s/t)>\sqrt{\epsilon_0}>0$.

We see that in the integral
$$
\int_{t_0/t}^1\lambda^{-2}I(\lambda)d\lambda
$$
$I(\lambda)$ is always bounded, because it can be controlled by
$$
\int_{|y|<1-\lambda}\frac{dy}{\sqrt{(1-\lambda)^2-|y|^2}} = (1-\lambda)\int_{|y|<1}\frac{dy}{\sqrt{1-|y|^2}}\leq C.
$$
The possible singularity comes from the term $\lambda^{-2}$. Thus to bound the above integra, the key is to estimate the asymptotic behavior of $I(\lambda)$ when $\lambda\rightarrow 0^+$.

In the case {\bf IA} and {\bf IB}, the function $I(\lambda)$ is bounded as following:
\begin{equation}\label{eq 1 09-12-2017}
I(\lambda) = \int_{D_{\lambda}}\frac{dy}{\sqrt{(1-\lambda)^2-|y|^2}}\leq S(D_1)\cdot \sup_{D_1}{((1-\lambda)^2-|y|^2)^{-1/2}}
\end{equation}
where $S(D_1)$ stands for the area of the disc $D_1$. Thus
\begin{equation}\label{eq 2 09-12-2017}
I(\lambda)\leq C\lambda^2\left(\frac{t-r+1}{2t}-\lambda\right)^{-1/2}.
\end{equation}
Thus
$$
\aligned
\int_{t_0/t}^1\lambda^{-2}I(\lambda)d\lambda\leq& \int_{t_0/t}^{\frac{t-r+1}{2t}} + \int_{\frac{t-r+1}{2t}}^1 \lambda^{-2}I(\lambda)d\lambda
\\
\leq& C\int_{t_0/t}^{\frac{t-r+1}{2t}}\left(\frac{t-r+1}{2t}-\lambda\right)^{-1/2}d\lambda + C\int_{\frac{t-r+1}{2t}}^1 \lambda^{-2}d\lambda
\\
\leq& C(\epsilon_0)
\endaligned
$$
here we applied the fact that
$$
\int_{\frac{t-r+1}{2t}}^1 \lambda^{-2}d\lambda\leq C\left(\frac{t-r}{t}\right)^{-2}\leq C(\epsilon_0).
$$
\end{proof}
%\marginpar{Verified.}

\begin{proof}[Proof of lemma \ref{lem 2 w-Kirchhoff}]
As explained, we need to estimate the decreasing rate of $I(\lambda)$ when $\lambda\rightarrow 0^+$. In this case we see that for $\lambda\in [t_0/t,1]$, we will have the following cases
$$
{\bf IA} \quad \frac{t-r+1}{2t}\quad {\bf IIA}\quad \frac{r+1}{t}\quad {\bf IIB}\quad \frac{t+r+1}{2t}\quad{\bf IIIB}\quad \text{for } r>\frac{t-1}{3}
$$
when $\frac{t-r}{t}$ is taken sufficiently small.

We give first a geometric description on $D_{\lambda}$ when $\lambda\in[\frac{t-r+1}{2t},\frac{1+t^{-1}}{2t}]$. When $\frac{t-r}{t}$ sufficiently small, this interval is contained in $[\frac{t-r+1}{2t},\frac{r+1}{t}]$. Thus we are in the case {\bf IIA}.

For the clarity of the presentation, let us introduce and recall some geometric notation: we denote by
$$
\aligned
&O: \text{the origin}
\\
&X: \text{the point} (-r/t,0),
\\
&X_+: \text{the point} (1,0)
\\
&X_-: \text{the point} (-1,0)
\\
&O_0: \text{the circle } \{y\in\RR^2||y| = 1-\lambda\}
\\
&O_1: \text{the circle centered at $X$ and of radius $\lambda-t^{-1}$}
\\
&P_+: \text{the upper point of intersection of the two circles } O_0 \text{ and } O_1.
\\
&P_-: \text{the lower point of intersection of the two circles } O_0 \text{ and } O_1.
\\
&D_{\lambda}:=D_0\cap D_1.
\endaligned
$$
We see that by the above parametrization of $D_\lambda$ with $(\rho,\theta)$, the bound of $\theta$, denoted by $\theta_0$, is attainted at $P_\pm$ (by angle $\angle X_+OP_+$ and $\angle X_+ O P_-$) or attained by the two tangent lines on the circle $O_1$ from $O$. For a point in $D_\lambda$, the variable $\theta$ varies from $\theta_0$ to $2\pi-\theta_0$.

Denote by $\beta = \angle XP_+O$, we observe that when $\beta\leq \pi/2$, $\theta_0$ is attained by the tangents from $O$ on the circle $O_1$, when $\beta>\pi/2$, $\theta_0$ is attained by $\angle X_+OP_-$. We remark that $\beta$ is increasing with respect to $\lambda$ on $[\frac{t-r+1}{2t}, \frac{1+t^{-1}}{2}]$. This is checked as following: in the triangle $\Delta XOP_+$
$$
\cos\beta = \frac{(\lambda - t^{-1})^2 + (1-\lambda)^2 - (r/t)^2}{2(\lambda-t^{-1})(1-\lambda)}.
$$
and
$$
\frac{ d\cos\beta}{d\lambda} = \frac{(2\lambda - (1+t^{-1}))\big((1-t^{-1})^2-(r/t)^2\big)}{2(\lambda-t^{-1})^2(1-\lambda)^2}\leq 0,
$$
where we remark that $r\leq t-1$ and $t^{-1}\leq t_0^{-1}\leq \lambda\leq 1$.

We denote by $\lambda_-$ the value of $\lambda$ for which the angle $\angle XP_+O$ attained $\pi/2$. By Pythagoras's theorem, we see that
$$
(\lambda_- - t^{-1})^2+(1-\lambda_-)^2 = (r/t)^2
$$
which leads to
$$
\aligned
\lambda_{-} =& \frac{1+t^{-1}-\sqrt{(1+t^{-1})^2-2(1-(r/t)^2+t^{-2})}}{2}
\\
=&(1-(r/t))\frac{1+(r/t) + t^{-1}(t-r)^{-1}}{(1+t^{-1}) + \sqrt{(1+t^{-1})^2-2(1-(r/t)^2+t^{-2})}}.
\endaligned
$$
When $(t-r)/t$ is sufficiently small (for example $(t-r)/t<1/10$), the above equation leads to
$$
\frac{1}{2}(1-(r/t))\leq \lambda_-\leq 3(1-(r/t)).
$$
We observe that when $\lambda\in[\lambda_-,(1+t^{-1})/2]$, $\theta_0$ is attainted by $\angle X_+OP_+$ thus (by the law of cosine)
$$
\cos\theta_0 = \frac{(\lambda-t^{-1})^2-(1-\lambda)^2-(r/t)^2}{2(1-\lambda)(r/t)}.
$$
We remark that $\theta_0$ is decreasing with respect to $\lambda$. We thus take $\lambda=(1+t^{-1})/2$ and see that $|P_+X| = |P_+O| = \frac{1-t^{-1}}{2}$ thus (remark that $\cos(\pi-\theta_0)$ is decreasing with respect to $\lambda$)
$$
\inf_{[\lambda_-,(1+t^{-1})/2]}\{\cos(\pi-\theta_0)\}  =  \frac{r}{t-1} .
$$
We see that when $\frac{t-r}{t}\leq \delta_0$, we have $\frac{t-1-r}{t-1}<\delta_0$ ($t\geq t_0\geq 2, r\leq r-1$), thus
\begin{equation}\label{eq 1 proof lem 2 w-Kirchhoff}
\inf_{[\lambda_-,(1+t^{-1})/2]}\{\cos(\pi - \theta_0)\}\geq 1-\delta_0.
\end{equation}
We will prove in the following lemma \ref{lem 3 w-Kirchhoff} that there exists $\frac{\pi}{2}>\alpha_0>0$ such that when $0\leq \pi-\theta_0\leq \alpha_0$,
\begin{equation}\label{eq 2 proof lem 2 w-Kirchhoff}
\pi-\theta_0 \leq 2\sqrt{1-\cos(\pi-\theta_0)}.
\end{equation}
Remark that in \eqref{eq 1 proof lem 2 w-Kirchhoff} we can make a choice of $\delta_0$ sufficiently small that $\pi-\theta_0\leq \alpha_0$, thus on $[\lambda_-,\frac{1+t^{-1}}{2}]$ \eqref{eq 2 proof lem 2 w-Kirchhoff} holds.

Now we begin to discuss on the interval $\lambda\in [t_0/t,1]$. We divide it into five pieces and on each we apply different estimates:

$\bullet$ On the interval $[t_0/t,\frac{t-r+1}{2t}]$. We are in the case {\bf IA}. In this case \eqref{eq 1 09-12-2017} and \eqref{eq 2 09-12-2017} hold as in the proof of lemma \ref{lem 1 w-Kirchhoff}. Then
$$
\int_{t/t_0}^{\frac{t-r+1}{2t}}\lambda^{-2}I(\lambda)d\lambda\leq C\int_{t/t_0}^{\frac{t-r+1}{2t}}\big(\frac{t-r+1}{2t}-\lambda\big)^{-1/2}d\lambda
\leq C\big(\frac{t-r}{t}\big)^{1/2}.
$$

$\bullet$ On the interval $[\frac{t-r+1}{2t},\lambda_-)$: in  this case the bound on $\theta$ is attended by the two tangents from $O$ on the circle $O_1$. In this case we observe by geometric relations that:
$$
\sin(\pi - \theta_0) = \frac{\lambda t}{r}\leq \frac{t}{r}\lambda_-\leq C\frac{t-r}{t}.
$$
when $\frac{t-r}{t}$ sufficiently small, the above relation leads to the bound
$\pi - \theta_0\leq C\lambda$. Recall that we are in the case {\bf IIA},
$$
I(\lambda) =2\int_{\rho_0}^{\rho_1}\frac{(\pi-\theta(\rho))\rho d\rho}{\sqrt{(1-\lambda)^2-\rho^2}}, \quad \rho_0 = \frac{r+1}{t}-\lambda,\quad \rho_1 = 1-\lambda
$$
and $0\leq \pi-\theta(\rho)\leq \pi-\theta_0$. Also recall that $\lambda_-\leq C(t-r)/t$, then
$$
\aligned
\int_{\frac{t-r+1}{2t}}^{\lambda_-}\lambda^{-2}I(\lambda)d\lambda
\leq& C\int_{\frac{t-r+1}{2t}}^{\lambda_-} \lambda^{-2}|\pi - \theta_0|\int_{\rho_0}^{\rho_1}\frac{\rho d\rho}{\sqrt{(1-\lambda)^2-\rho^2}}
\\
\leq& C\int_{\frac{t-r+1}{2t}}^{C(t-r)t^{-1}}\lambda^{-1}\int_{\frac{r+1}{t}-\lambda}^{1-\lambda}\frac{\rho d\rho}{\sqrt{(1-\lambda)^2-\rho^2}}
\\
\leq& C\left(\frac{t-r}{t}\right)^{1/2}\int_{\frac{t-r+1}{2t}}^{C(t-r)t^{-1}}\lambda^{-1}d\lambda
\\
\leq&  C\left(\frac{t-r}{t}\right)^{1/2}.
\endaligned
$$
%\marginpar{Verified}

$\bullet$ On the interval $[\lambda_-,(1+t^{-1})/2]$, as explained above, when $\frac{t-r}{t}$ is sufficiently small, the angle $\pi - \theta_0$ is bounded by $\alpha_0$ for $\lambda_-\leq \lambda\leq (1+t^{-1})/2$ so
$$
\pi - \theta_0\leq 2\sqrt{1+\cos\theta_0} = 2\sqrt{\frac{(\lambda-t^{-1})^2-(1-\lambda - r/t)^2}{2(1-\lambda)(r/t)}}
$$
thus
$$
\pi-\theta_0\leq C\left(\lambda - \frac{t-r+1}{2t}\right)^{1/2}\left(\frac{t-r-1}{t}\right)^{1/2}
$$
Thus we see that
$$
\int_{\lambda_-}^{\frac{1+t^{-1}}{2}}\lambda^{-2}I(\lambda)d\lambda
\leq \int_{\lambda_-}^{\frac{1+t^{-1}}{2}}2\lambda^{-2}(\pi-\theta_0)\int_{\rho_0}^{\rho_1}\frac{\rho d\rho}{\sqrt{(1-\lambda)^2-\rho^2}}d\lambda
$$
We remark that in this interval $\lambda\in[\lambda_-,\frac{1+t^{-1}}{2}]$, we are in the case {\bf IIA} (when $\frac{t-r}{t}$ is taken sufficiently small) $\rho_1=1-\lambda$ and $\rho_0 = \frac{r+1}{t}-\lambda$, thus
$$
\aligned
\int_{\lambda_-}^{\frac{1+t^{-1}}{2}}\lambda^{-2}I(\lambda)d\lambda
\leq& C\int_{\lambda_-}^{\frac{1+t^{-1}}{2}}\lambda^{-2}(\pi-\theta_0)\int_{\frac{r+1}{t}-\lambda}^{1-\lambda}\frac{\rho d\rho}{\sqrt{(1-\lambda)^2-\rho^2}}d\lambda
\\
\leq& C\left(\frac{t-r}{t}\right)\int_{\lambda_-}^{\frac{1+t^{-1}}{2}} \lambda^{-2}\left(\lambda - \frac{t-r+1}{t}\right)^{1/2}d\lambda
\endaligned
$$
where we remark that
$$
\int_{\frac{r+1}{t}-\lambda}^{1-\lambda}\frac{\rho d\rho}{\sqrt{(1-\lambda)^2-\rho^2}} \leq C \left(\frac{t+r-1}{t}\right)^{1/2}.
$$
Remark that see that $\lambda-\frac{t-r-1}{2t}\leq \lambda$, then
$$
\int_{\lambda_-}^{\frac{1+t^{-1}}{2}}\lambda^{-2}I(\lambda)d\lambda\leq C\left(\frac{t-r}{t}\right)\int_{\lambda_-}^{\frac{1+t^{-1}}{2}}\lambda^{-3/2}d\lambda
$$
which leads to
$$
\int_{\lambda_-}^{\frac{1+t^{-1}}{2}}\lambda^{-2}I(\lambda)d\lambda\leq C\lambda_-^{-1/2}\left(\frac{t-r}{t}\right).
$$
We recall that $\lambda_-\geq C\frac{t-r}{t}$ thus
$$
\int_{\lambda_-}^{\frac{1+t^{-1}}{2}}\lambda^{-2}I(\lambda)d\lambda\leq C\left(\frac{t-r}{t}\right)^{1/2}.
$$
%\marginpar{Verified}

$\bullet$ On the interval $(\frac{1+t^{-1}}{2},\frac{r+1}{t}]$, we are in case {\bf IIA}. Thus $\rho_0 = \frac{r+1}{t}-\lambda$, $\rho_1=1-\lambda$. Then
$$
\int_{(1+t^{-1})/2}^{\frac{r+1}{t}}\lambda^{-2}I(\lambda)d\lambda\leq C\int_{(1+t^{-1})/2}^{\frac{r+1}{t}}\int_{\frac{r+1}{t}-\lambda}^{1-\lambda}\frac{d\rho}{\sqrt{(1-\lambda)^2-\rho^2}}d\lambda\leq C\big(\frac{t-r}{t}\big)^{1/2}.
$$

$\bullet$ On the interval $(\frac{r+1}{t},1]$, we are in the case {\bf IIB} and {\bf IIIB}. Thus $\rho_0=0, \rho_1=1-\lambda$. Then
$$
\int_{\frac{1+r}{t}}^1\lambda^{-2}I(\lambda)d\lambda\leq C\int_{\frac{1+r}{t}}^1I(\lambda)d\lambda
\leq C\int_{\frac{1+r}{t}}^1\int_0^{1-\lambda}\frac{d\rho}{\sqrt{(1-\lambda)^2-\rho^2}}d\lambda\leq C\big(\frac{t-r}{t}\big)^{1/2}.
$$
%\marginpar{proof verified}

%$\bullet$ On the interval $[1/2,\lambda_+]$, we remark that
%$$
%\int_{1/2}^{\lambda_+}\lambda^{-2}I(\lambda)d\lambda\leq C\int_{1/2}^{\lambda_+}I(\lambda)d\lambda\leq C\int_{1/2}^{\lambda_+}\int_{\rho_0}^{\rho_1}\frac{\rho d\rho}{\sqrt{(1-\lambda)^2-\rho^2}}d\lambda
%$$
%which leads to (recall $\rho_0 = \frac{r+1}{t}-\lambda$ and $\rho_1=1-\lambda$)
%$$
%\int_{1/2}^{\lambda_+}\lambda^{-2}I(\lambda)d\lambda\leq C\left(\frac{t-r}{t}\right)^{1/2}.
%$$
%
%
%On the interval $[\lambda_+,1]$. We remark that $1-\lambda_+\leq C\frac{t-r}{t}$. This leads to
%$$
%\aligned
%\int_{t_0/t}^{\lambda_-}\lambda^{-2}I(\lambda)d\lambda
%\leq& C\int_{\lambda_+}^1\int_{D_0(\lambda)}\frac{\rho d\rho}{\sqrt{(1-\lambda)^2 - \rho^2}}
%\\
%=&C\int_{\lambda_+}^1\int_0^{1-\lambda}\frac{\rho d\rho}{\sqrt{(1-\lambda)^2 - \rho^2}}
%\endaligned
%$$
%Recall that
%$$
%\int_0^{1-\lambda}\frac{\rho d\rho}{\sqrt{(1-\lambda)^2 - \rho^2}} = 1-\lambda.
%$$
%Then
%$$
%\int_{t_0/t}^{\lambda_-}\lambda^{-2}I(\lambda)d\lambda\leq C(1-\lambda_+)\leq C\left(\frac{t-r}{t}\right)^{1/2}
%$$
\end{proof}

\begin{lemma}\label{lem 3 w-Kirchhoff}
There exists a constant $0<\alpha_0<\frac{\pi}{2}$ such that for all $0\leq \alpha\leq \alpha_0$, the following relation holds
$$
\alpha\leq 2\sqrt{1-\cos\alpha}.
$$
\end{lemma}
\begin{proof}
We define the function $f:[0,\pi/2]\rightarrow \RR$ by
$$
f(\alpha) =
\left\{
\aligned
&\frac{\sqrt{1-\cos\alpha}}{\alpha},\quad \alpha>0
\\
&\frac{\sqrt{2}}{2},\quad \alpha = 0
\endaligned
\right.
$$
Then we see that $f$ is continuous. So we see that there exists a constant $\alpha_0>0$ such that $f(\alpha)\geq \frac{1}{2}$.
\end{proof}
%\marginpar{verified}

\section{Hessian form of wave component}\label{sec other-tools}
As we explained in introduction, we need to establish more precise  bound on the Hessian form of the wave component and refined estimate of Klein-Gordon component near the light-cone. Both of them are based on semi-hyperboloidal decomposition of the wave operator (which are explained in \cite{M1}). Here we just give a sketch. We first recall that
\begin{equation}\label{eq 1 wave-hessian}
(s/t)^2\del_t\del_t u + \frac{2x^a}{t}\delu_a \del_t u - \sum_a\delu_a\delu_a u - \frac{r^2}{t^3}\del_tu + \frac{3}{t}\del_tu = \Box u.
\end{equation}
Let us consider the wave equation
$$
\Box u = f,
$$
then \eqref{eq 1 wave-hessian} becomes
\begin{equation}\label{eq 2 wave-hessian}
(s/t)^2\del_t\del_t u = f - R_1[u]
\end{equation}
where
$$
R_1[u]:=t^{-1}\left( \frac{2x^a}{t}L_a\del_t u - \sum_aL_a\delu_a u - \frac{r^2}{t^2}\del_tu + 3\del_tu\right)
$$

In our case, $f$ is a quadratic form thus has better decay. The rest term has an extra decay factor $t^{-1}$. Thus \eqref{eq 2 wave-hessian} indicates that $(s/t)^2\del_t\del_t u$ has better decay than $\del_\alpha u$.

Furthermore, we remark that
\begin{equation}\label{eq 3 wave-hessian}
\del_a\del_t = \delu_a\del_t - \frac{x^a}{t}\del_t\del_t,\quad \del_a\del_b = \delu_a\delu_b - \Big(\frac{x^b}{t}\delu_a\del_t + \frac{x^a}{t}\del_t\delu_b\Big) + \frac{x^ax^b}{t^2}\del_t\del_t - \frac{\delta_{ab}}{t}\del_t.
\end{equation}
%\marginpar{verified}

%Then we consider the Klein-Gordon component:
%$$
%\Box v + v = g
%$$
%then \eqref{eq 1 wave-hessian}
%\begin{equation}\label{eq 1 Kg-s/t}
%v = g - (s/t)^2\del_t\del_t v - R_1[v].
%\end{equation}
%Form this identity we see that $v$ is expressed by the linear combination of a quadratic term $g$, a fast decreasing term $R_1[v]$ and a term $(s/t)^2\del_t\del_t v$. Thus we can see that $v\sim (s/t)^2\del_t\del_t v$. This leads to a fast decay rate near $\del \Kcal$ where $(s/t)$ becomes non-trivial.

\section{Initialization of the bootstrap argument}\label{sec bootstrap}
\subsection{Bootstrap argument and energy bounds}
In this section we initialize the bootstrap argument. We remark that
$$
E(2,\del^IL^J u)^{1/2},\quad
E_1(2,\del^IL^J v)^{1/2}
$$
are determined by the initial data
$$
u_0,u_1,v_0,v_1.
$$
Let $\vep$ be a positive constant, sufficiently small, and suppose that $\forall |I|\leq N, N\geq 11$ and $ \alpha =0,1,2,3$,
$$
\aligned
\|\del_{\alpha}\del^I u_0\|_{L^2(\Hcal_2)} + \|\del^I u_1\|_{L^2(\Hcal_2)}\leq& \vep,
\\
\|\del_{\alpha}\del^I v_0\|_{L^2(\Hcal_2)} + \|\del^I v_0\|_{L^2(\Hcal_2)} + \|\del^I v_1\|_{L^2(\Hcal_2)}\leq& \vep.
\endaligned
$$
We observe that when $\vep$ sufficiently small (i.e. $\exists \vep_0>0, \forall 0\leq \vep\leq \vep_0$), there exists a positive constant $C_0$ such that
$$
E(2,\del^IL^J u)^{1/2} \leq C_0\vep,\quad E_1(2,\del^IL^J v)^{1/2}\leq C_0\vep, \text{ for } |I|+|J|\leq N
$$
where $C_0$ is a constant determined by $N$ and the system. We denote by $[2,T)$ the maximal (hyperbolic) time interval of existence of the local solution associated to $(u_0,u_1,v_0,v_1)$. We will prove that $T=\infty$. To do so, let us suppose that $T<\infty$.

We suppose that on a time interval $[2,s_1]\subset[2,T)$, the following energy bounds hold:
\begin{equation}\label{eq 1 bootstrap}
E(s,\del^IL^J u)^{1/2} \leq C_1\vep s^{\delta},\quad E_1(s,\del^IL^J v)^{1/2}\leq C_1\vep s^{\delta},\quad |I|+|J|\leq N.
\end{equation}
In these bounds, $(C_1,\vep)$ is a pair of positive constants, satisfying $C_1\vep\leq 1$.  $0<\delta<1/10$ is a fixed constant. We denote by
$$
s^*:=\sup_{T>s_1\geq 2}\{\text{\eqref{eq 1 bootstrap} holds on}[2,s_1]\}.
$$
When choosing $C_1>C_0$, by continuity, $s^*>2$.

Suppose that we can establish the following {\bf refined} energy bounds on the same interval
\begin{equation}\label{eq 2 bootstrap}
E(s,\del^IL^J u)^{1/2}\leq \frac{1}{2}C_1\vep s^{\delta},\quad E_1(s,\del^IL^J v)^{1/2}\leq \frac{1}{2}C_1\vep s^{\delta},\quad |I|+|J|\leq N.
\end{equation}
By continuity, if $s^*<T$, we see that at the time $s^*$, at least one of the inequality in \eqref{eq 1 bootstrap} becomes equality, and this contradicts the above \eqref{eq 2 bootstrap}. So we conclude that $s^*=T$. Remark that $T$ is the maximal time of existence, and we see that
$$
\lim_{s\rightarrow T^-}\big(E(s,\del^IL^J u)^{1/2} + E_1(s,\del^IL^J v)^{1/2}\big) <\infty,\quad |I|+|J|\leq N.
$$
By standard local theory, we see that this is not possible when $T$ is the maximal time of existence (one can construct a local solution form the time $T-\eta$ and extends the local solution on a larger interval). Thus $T = \infty$ which guarantees theorem \ref{thm main}. The rest of this article is devoted to the proof of \eqref{eq 2 bootstrap} based on \eqref{eq 1 bootstrap}.

We will first  establish some basic bounds based on \eqref{eq 1 bootstrap}. They are basic $L^2$ bounds and $L^{\infty}$ bounds. We state them in the following two subsections.
\subsection{Basic $L^2$ bounds}
These are direct from the expression of the energy. We see that for $|I|+|J|\leq N$:
\begin{equation}\label{eq 1 L2-basic}
\|(s/t)\del_{\alpha}\del^IL^J u\|_{L^2(\Hcal_s)}\leq CC_1\vep s^{\delta},\quad \|(s/t)\del_{\alpha}\del^IL^J v\|_{L^2(\Hcal_s)}\leq CC_1\vep s^{\delta}.
\end{equation}
\begin{equation}\label{eq 2 L2-basic}
\|\delu_a\del^IL^J u\|_{L^2(\Hcal_s)}\leq CC_1\vep s^{\delta},\quad \|\delu_a\del^IL^J v\|_{L^2(\Hcal_s)}\leq CC_1\vep s^{\delta}
\end{equation}
and
\begin{equation}\label{eq 3 L2-basic}
\|\del^IL^J v\|_{L^2(\Hcal_s)}\leq CC_1\vep s^{\delta}.
\end{equation}
For $|I|+|J|\leq N-1$
\begin{equation}\label{eq 4 L2-basic}
\|\del_\alpha\del^IL^J v\|_{L^2(\Hcal_s)}\leq CC_1\vep s^{\delta}.
\end{equation}

Then by \eqref{ineq 1.1 comm}, we see that for $|I|+|J|\leq N$,
\begin{equation}\label{eq 5 L2-basic}
\|(s/t)\del^IL^J \del_{\alpha} u\|_{L^2(\Hcal_s)}\leq CC_1\vep s^{\delta},\quad \|(s/t)\del^IL^J\del_{\alpha}v\|_{L^2(\Hcal_s)}\leq CC_1\vep s^{\delta}.
\end{equation}
By \eqref{ineq 2.1 comm} and \eqref{eq 4 L2-basic}, for $|I|+|J|\leq N-1$:
\begin{equation}\label{eq 6 L2-basic}
\|t\del^IL^J \delu_a v\|_{L^2(\Hcal_s)} + \|t\delu_a\del^IL^J v\|_{L^2(\Hcal_s)}\leq CC_1\vep s^{\delta}.
\end{equation}
%\marginpar{Verified}

\subsection{Basic $L^\infty$ bounds}\label{subsec decay-basic}
By the Klainerman-Sobolev inequality \eqref{ineq 1 KS}, \eqref{ineq 1.1 comm}, \eqref{eq 1 lem 2 homo} and the above basic $L^2$ bounds, the following $L^{\infty}$ bounds are immediate. In $\Kcal$, for $|I|+|J|\leq N-2$:
\begin{equation}\label{ineq 1 Linfty-basic}
|\del^IL^J\del_{\alpha}u|\leq CC_1\vep s^{-1+\delta},\quad |\del^IL^J\del_{\alpha}v|\leq CC_1\vep s^{-1+\delta}.
\end{equation}
\begin{equation}\label{ineq 2 Linfty-basic}
|\del^IL^J \delu_a u|\leq CC_1\vep t^{-1}s^{\delta},\quad |\del^IL^J \delu_a v|\leq CC_1\vep t^{-1}s^{\delta}.
\end{equation}
\begin{equation}\label{ineq 3 Linfty-basic}
|\del^IL^J v|\leq CC_1\vep t^{-1}s^{\delta}.
\end{equation}

And by \eqref{ineq 2.1 comm} and the above bound \eqref{ineq 3 Linfty-basic}, for $|I|+|J|\leq N-3$,
\begin{equation}\label{ineq 4 Linfty-basic}
|\del^IL^J\delu_a v|\leq CC_1\vep t^{-2}s^{\delta}.
\end{equation}

Furthermore, we also have, for $|I|+|J|\leq N-3$,
\begin{equation}\label{ineq 1 Linfty-rad}
|\delu_a\del^IL^J u|\leq CC_1\vep (s/t)^2s^{-1+\delta}.
\end{equation}
To see this, we first remark that
\begin{equation}\label{ineq 2 Linfty-rad}
\aligned
\left|\del_r\delu_a\del^IL^J u\right| =& \left|(x^b/r)\del_b\left(t^{-1}L_a\del^IL^Ju\right)\right| \leq t^{-1}\sum_b\left|\del_bL_a\del^IL^J u\right|
\\
\leq& CC_1\vep t^{-1}s^{-1+\delta}\sim CC_1\vep (t-r)^{-1/2+\delta/2}t^{-3/2+\delta/2}
\endaligned
\end{equation}
where in the last inequality we have applied \eqref{ineq 1.1 comm} combined with the above $L^{\infty}$ bounds \eqref{ineq 1 Linfty-basic}. Then we integrate \eqref{ineq 2 Linfty-rad} along the ray radial ray and see that \eqref{ineq 1 Linfty-rad} is established.
%\marginpar{Strictly verified till here.}

\subsection{Basic $L^{\infty}$ bound on Null quadratic form}
%\marginpar{New subsection added}
We consider the $L^{\infty}$ bound of the following term:
$$
\del^IL^J\left(P^{\alpha\beta}\del_{\alpha}u\del_{\beta}u\right)
$$
where we recall that $P^{\alpha\beta}$ is a null quadratic form.
Recall that
$$
P^{\alpha\beta}\del_{\alpha}u\del_{\beta}u = \Pu^{\alpha\beta}\delu_{\alpha}u\delu_{\beta}u = \Pu^{00}\del_tu\del_tu + \Pu^{a0}\delu_a\del_tu + \Pu^{0a}\del_tu\delu_au + \Pu^{ab}\delu_au\delu_bu
$$
and
$$
\del^IL^J\big(\Pu^{\alpha\beta}\delu_{\alpha}u\delu_{\beta}u\big) = \sum_{I_1+I_2+I_3=I\atop J_1+J_2+J_3=J}\del^{I_1}L^{J_1}\Pu^{\alpha\beta}\cdot \del^{I_2}L^{J_2}\delu_{\alpha}u\cdot\del^{I_3}L^{J_3}\delu_{\beta}u.
$$

Now we recall the following property for null quadratic form (see \cite{LM1} proposition 4.1.1): in $\Kcal$,
\begin{equation}\label{eq 1 null}
\big|\del^IL^J\Pu^{00}\big|\leq C(I,J)(s/t)^2.
\end{equation}
Also recall that $\Pu^{\alpha\beta}$ are homogeneous of degree zero, thus in $\Kcal$
$$
\big|\del^IL^J\Pu^{\alpha\beta}\big|\leq C(I,J)t^{-|I|}\leq C(I,J).
$$
Then recall the $L^{\infty}$ bounds \eqref{ineq 1 Linfty-basic} and \eqref{ineq 1 Linfty-rad} (combined with commutator \eqref{ineq 2.2 comm}) and substitute these bounds into the expression of $\del^IL^J\big(P^{\alpha\beta}\del_{\alpha}u\del_{\beta}u\big)$, we obtain that for $|I|+|J|\leq N-3$,
\begin{equation}\label{ineq 1 Linfty-null}
\big|\del^IL^J\big(P^{\alpha\beta}\del_{\alpha}u\del_{\beta}u\big)\big|\leq C(I,J)(C_1\vep)^2(s/t)^2s^{-2+2\delta}.
\end{equation}

\section{Estimates on Hessian form of wave component}\label{sec hessian}
This is based on the combination of \eqref{eq 2 wave-hessian} together with the basic $L^{\infty}$ bounds established in the above section. We derive the wave equation with respect to $\del^IL^J$ and see that
\begin{equation}\label{eq 3 w-hessian}
\Box \del^IL^J u = \del^IL^J \left(f(v,\del v)\right)
\end{equation}
where
$$
f(v,\del v) : = v\left(A_1^{\alpha}\del_{\alpha}v + R v\right) + Q^{\alpha\beta}\del_{\alpha}v\del_{\beta}v
$$
Thus by \eqref{eq 2 wave-hessian}
\begin{equation}\label{eq 4 w-hessian}
(s/t)^2\del_t\del_t \del^IL^Ju = \del^IL^J\left(f(v,\del v)\right) - R_1[\del^IL^J u].
\end{equation}
Then we state the following bounds:
\begin{lemma}\label{lem 1 w-hessian}
Under the bootstrap assumption \eqref{eq 1 bootstrap}, in $\Kcal$ for $|I|+|J|\leq N-3$
\begin{equation}\label{eq 1 lem 1 w-hessian}
|\del^IL^J(f(v,\del v))|\leq C(C_1\vep)^2(s/t)^2s^{-2+2\delta},
\end{equation}
and for $|I|+|J|\leq N-3$
\begin{equation}\label{eq 2 lem 1 w-hessian}
|R_1[\del^IL^J u]|\leq CC_1\vep (s/t)s^{-2+\delta}.
\end{equation}
\end{lemma}
\begin{proof}
For the bound of $f$, we see that it is a bilinear form of $v$ and $\del v$ thus when we derive $f$ with respect to $\del^IL^J$, we see that $\del^IL^J f$ is a linear combination of quadratic terms of the following form:
$$
\del^{I_1}L^{J_1}v\del^{I_2}L^{J_2}v,\quad \del^{I_1}L^{J_1}\del_{\alpha}v\del^{I_2}L^{J_2}\del_{\beta}v,\quad \del^{I_1}L^{J_1}v\del^{I_2}L^{J_2}\del_{\alpha}v.
$$
We see that by the basic $L^{\infty}$ bounds, the above terms are bounded by $C(C_1\vep)^2t^{-2}s^{2\delta}$, thus \eqref{eq 1 lem 1 w-hessian} is established.

For $R_1$, we see that
$$
R_1[\del^IL^Ju]:=t^{-1}\left( \frac{2x^a}{t}L_a\del_t \del^IL^Ju - \sum_aL_a\delu_a \del^IL^Ju - \frac{r^2}{t^2}\del_t\del^IL^Ju + 2\del_t\del^IL^Ju\right)
$$
We remark that there is a decreasing $t^{-1}$ factor.  So when $|I|+|J|\leq N-3$, by the estimates of commutators \eqref{ineq 1.1 comm}, \eqref{ineq 2.1 comm} and \eqref{ineq 2.2 comm}, and the basic $L^{\infty}$ bounds established in subsection \ref{subsec decay-basic}, we see that \eqref{eq 2 lem 1 w-hessian} is established.
\end{proof}

Then we establish the following bound by \eqref{eq 4 w-hessian}: for $|I|+|J|\leq N-3$,%\marginpar{$|C_1\vep|\leq 1$}
\begin{equation}\label{eq 5 w-hessian}
|\del_t\del_t \del^IL^J u|\leq CC_1\vep (s/t)^{-1}s^{-2+2\delta}.
\end{equation}

We remark that
$$
\delu_a\del_\alpha\del^IL^Ju = t^{-1}L_a\del_\alpha\del^IL^Ju
$$
and combined by the estimates of commutators \eqref{ineq 1.1 comm} and the fact that $\del_b$ is a linear combination of $\del_{\alpha}$ with homogeneous coefficients of degree zero,
\begin{equation}\label{eq 6 w-hessian}
|\delu_a\del_t \del^IL^J u| + |\del_t\delu_a \del^IL^J u| + |\delu_a\delu_b\del^IL^Ju| \leq CC_1\vep(s/t)s^{-2+\delta}.
\end{equation}
Thus by \eqref{eq 3 wave-hessian} and basic decay bounds,
$$
|\del_{\alpha}\del_{\beta}\del^IL^Ju|\leq CC_1\vep (s/t)^{-1}s^{-2+2\delta}.
$$
Then by the estimates of commutator \eqref{ineq 1.1 comm},
\begin{equation}\label{eq 7 w-hessian}
|\del^IL^J\del_{\alpha}\del_{\beta}u|\leq CC_1\vep (s/t)^{-1}s^{-2+2\delta},\quad |I|+|J|\leq N-3.
\end{equation}
\section{$L^{\infty}$ estimate on Klein-Gordon component}\label{sec KG-infty}
\subsection{Technical preparations}
We recall \eqref{eq 1 KG-infty}
$$
\Box w + w = -P^{\alpha\beta}\left(2\del_{\alpha}u\del_{\beta}\Box u + m^{\alpha'\beta'}\del_{\alpha'}\del_{\alpha}u\del_{\beta'}\del_{\beta}u\right)
$$
which leads to
\begin{equation}\label{eq 2 Kg-infty}
\Box w + w =- P^{\alpha\beta}m^{\alpha'\beta'}\del_{\alpha'}\del_{\alpha}u\del_{\beta'}\del_{\beta}u -P^{\alpha\beta}\left(2\del_{\alpha}u\del_{\beta}\left(f(v,\del v)\right)\right)
\end{equation}
and we derive the above equation with respect to $\del^IL^J$:
\begin{equation}\label{eq 3 Kg-infty}
\Box \del^IL^J w + \del^IL^J w = \del^IL^J B + \del^IL^J T
\end{equation}
with bilinear term
$$
B := P^{\alpha\beta}m^{\alpha'\beta'}\del_{\alpha'}\del_{\alpha}u\del_{\beta'}\del_{\beta}u
$$
and trilinear term
$$
T(\del u,v,\del v,\del\del v) := P^{\alpha\beta}\del_{\alpha}u \cdot \del_{\beta}\left(v(A^{\gamma}\del_{\gamma}v + Rv) + Q^{\alpha'\beta'}\del_{\alpha'}v\del_{\beta'}v\right)
$$

The quadratic term $B$ is a null term, we make the following null decomposition:
$$
\aligned
B =& P^{\alpha\beta}m^{\alpha'\beta'}\del_{\alpha'}\del_{\alpha}u\del_{\beta'}\del_{\beta}u
= P^{\alpha\beta}\minu^{\alpha'\beta'}\delu_{\alpha'}\del_{\alpha}u\delu_{\beta'}\del_{\beta}u
\\
=& P^{\alpha\beta}\minu^{00}\del_t\del_{\alpha}u\del_t\del_{\beta}u + 2P^{\alpha\beta}\minu^{a0}\delu_a\del_{\alpha}u\del_t\del_{\beta}u
+P^{\alpha\beta}\minu^{ab}\delu_a\del_{\alpha}u\delu_b\del_{\beta}u
\\
=& B^{00} + 2B^{a0} + B^{ab}
\endaligned
$$
with
$$
B^{\alpha'\beta'}:=P^{\alpha\beta}\minu^{\alpha\beta}\delu_{\alpha'}\del_{\alpha}u\delu_{\beta'}\del_{\beta}u.
$$

Now we recall the $L^{\infty}$ estimate of Klein-Gordon equation. This is firstly applied in \cite{Kl1}(Klainerman) in $3+1$ space-time dimension and then applied in many context. This is based on the following identity:
\begin{equation}\label{eq 3 KG-linfty}
s^{-1}\left(\del_s + (x^a/s)\delu_a\right)^2(sw) = \Box w + S_1[w]
\end{equation}
with $\del_s := (s/t)\del_t$ and
$$
S_1[w] :=  \sum_a\delu_a\delu_aw  + \frac{x^ax^b}{s^2}\delu_a\delu_b w + \frac{2x^a}{s^2}\delu_a w.
$$
%\marginpar{verified!}
Combine this with \eqref{eq 2 Kg-infty}, we see that
\begin{equation}\label{eq 4 KG-linfty}
\left(\del_s + (x^a/s)\delu_a\right)^2(s\del^IL^J w) + s\del^IL^J w = sS_1[\del^IL^J w] + s\del^IL^J\left(B + T\right).
\end{equation}
Taking a point $(t_1,x_1)\in \Kcal_{[2,s_1]}$ with $s_1^2 = t_1^2-r_1^2$, we denote by $\gamma(s;t_1,x_1)$ the integral curve of $\del_s + (x^a/s)\delu_a$ passing $(t_1,x_1)$ with $\gamma(s_1;t_1,x_1) = (t_1,x_1)$. We remark that for a point $(t,x)\in \gamma(s,t_1,x_1)$, the quantity $(r/t)$ is constant (with respect to $s$):
$$
\left(\del_s + (x^a/s)\delu_a\right)(r/t) = s^{-1}\left(t\del_t + r\del_r\right)(r/t) = 0.
$$
That is, $\gamma$ is the segment from the origin to $(t_1,x_1)$. We denote by $\gamma(s_0;t_1,r_1)$ the point at which $\gamma(s;t_1,r_1)$ enters the region $\Kcal_{[2,s_1]}$. Then we remark that for $(t_1,x_1)\in\Kcal_{[2,s_1]}$ with $r_1/t_1>3/5$, $\gamma(s;t_1,x_1)$ enters $\Kcal_{[2,s_1]}$  by intersecting the conical boundary $\del\Kcal := \{t=r+1\}$ at $s_0 = \sqrt{\frac{t_1+r_1}{t_1-r_1}}$ (that is, for $s_0\leq s\leq s_1$, $\gamma(s;t_1,r_1)\in \Kcal_{[2,s_1]}$) and when $r_1/t_1\leq 3/5$, $\gamma$ enters $\Kcal_{[2,s_1]}$ by intersecting $\Hcal_2$ at $s_0=2$.

Denote by $W_{t_1,x_1}(s) := w|_{\gamma(s;t_1,x_1)}$. We see that \eqref{eq 4 KG-linfty} leads to
$$
W_{t_1,x_1}''(s) + W_{t_1,x_1} (s) = F_{t,x}(s)
$$
with
$$
F_{t_1,x_1}(s): = s\left(S_1[\del^IL^J w] + \del^IL^J\left(B + T\right)\right)|_{\gamma(s;t_1,x_1)}
$$
Then we integrate the above ODE on the interval $[s_0,s_1]$. By standard ODE argument we see that
\begin{equation}\label{eq 5 KG-linfty}
|W_{t_1,x_1}(s_1)| + |W'_{t_1,x_1}(s_1)|\leq C\left(|W_{t_1,x_1}(s_0)| + |W'_{t_1,x_1}(s_0)|\right) + C\int_{s_0}^{s_1}|F_{t_1,x_1}(s)|ds.
\end{equation}
We remark that
$$
\aligned
W'_{t,x}(s) =& \left(\del_s + (x^a/s)\delu_a\right)\left(s\del^IL^J w(t,x)\right)
\\
=&\del^IL^J w(t,x) + s \left(\del_s + (x^a/s)\delu_a\right)\del^IL^J w(t,x).
\endaligned
$$
So above estimate can be translated into the following form:
\begin{equation}\label{eq 6 KG-linfty}
|s\del^IL^J w(t,x)| + \big|s\left(\del_s + (x^a/s)\delu_a\right)\left(\del^IL^J w\right)(t,x)|\leq
\left\{
\aligned
&CC_0\vep + \int_2^s|F_{t,x}(\lambda)|d\lambda,\quad r/t\leq 3/5,
\\
&\int_{s_0}^s|F_{t,x}(\lambda)|d\lambda,\quad 1>r/t>3/5
\endaligned
\right.
\end{equation}
with $s_0 = \sqrt{\frac{t+r}{t-r}}$. Here we applied the fact that when $\gamma(s_0;t,x)\in \del\Kcal$, $W_{t,x}(s_0) = W_{t,x}'(s_0) = 0$ and when $\gamma(s_0;t,x)\in\Hcal_2$, $W_{t,x}(s_0), W_{t,x}'(s_0)$ are determined by the initial data thus bounded by $CC_0$.  We recall the relation
$$
s^2t^{-1}\big|\del_tw(s,x)\big| \leq  \big|s\left(\del_s + (x^a/s)\delu_a\right)w(t,x)| + |L_aw(t,x)|
$$
So
\begin{equation}\label{eq 6' KG-linfty}
\aligned
|s\del^IL^J w(t,x)| + (s/t)|s\del_t \del^IL^J w(t,x)|\leq
\left\{
\aligned
&CC_0\vep + |L_a\del^IL^Jw(t,x)| + \int_2^s|F_{t,x}(\lambda)|d\lambda,\quad r/t\leq 3/5,
\\
&|L_a\del^IL^Jw(t,x)| + \int_{s_0}^s|F_{t,x}(\lambda)|d\lambda,\quad 1>r/t>3/5.
\endaligned
\right.
\endaligned
\end{equation}

The rest task is to bound the right-hand-side $F_{t,x}$. This will be done in the following subsection.

\subsection{The $(s/t)t^{-1}$ bound on $v$ and $\del v$}
We remark that $F_{t_1,x_1}\sim s\del^IL^J(B+T) + sS_1[\del^IL^J w]$. The essential work is to bound $\del^IL^JB$ and $\del^IL^J T$ and $S_1[\del^IL^J w]$. These are concluded in the following two lemmas:
\begin{lemma}\label{lem 1 KG-linfty}
Under the bootstrap assumption \eqref{eq 1 bootstrap}, the following bounds holds for $|I|+|J|\leq N-3$:
\begin{equation}
|\del^IL^J B| \leq C(C_1\vep)^2 (s/t)s^{-4+4\delta} ,\quad |\del^IL^J T|\leq C(C_1\vep)^3 (s/t)^2s^{-3+3\delta}.
\end{equation}
\end{lemma}
\begin{proof}
We recall that $\minu^{00} = \frac{s^2}{t^2}$. Remark that for $B^{00}$,
$$
\aligned
&\del^IL^J\big((s/t)^2\big)\del_t\del_{\alpha}u\cdot\del_t\del_{\beta}u\big)
\\
=& \del^{I_1}L^{J_1}\big((s/t)^2\big)\cdot \del^{I_2}L^{J_2}\del_t\del_{\alpha}u\cdot \del^{I_2}L^{J_2}\del_t\del_{\beta}u
\\
=& 2(s/t)\del^{I_1}L^{J_1}(s/t)\cdot \del^{I_2}L^{J_2}\del_t\del_{\alpha}u\cdot \del^{I_2}L^{J_2}\del_t\del_{\beta}u
\endaligned
$$
Recalling \eqref{eq 1 lem 2 homo} and the fact that in $\Kcal$, $s^{-1}\leq s/t$, we remark that by \eqref{eq 7 w-hessian}
$$
\aligned
\big|\del^IL^J\big((s/t)^2\big)\del_t\del_{\alpha}u\cdot\del_t\del_{\beta}u\big) \big|\leq C(s/t)^2\big|\del^{I_2}L^{J_2}\del_t\del_{\alpha}u\cdot \del^{I_2}L^{J_2}\del_t\del_{\beta}u\big|\leq C(C_1\vep)^2s^{-4+4\delta}.
\endaligned
$$
For the term $B^{a0}$ and $B^{aa}$, the estimates are similar, we apply the fact that $\minu^{a0}$ and $\minu^{ab}$ are homogeneous of degree zero, and \eqref{eq 6 w-hessian}.

For the term $T$, we see that it is a linear combination of the trilinear terms. In each term, there is a factor $\del_{\alpha}u$ and a quadratic form of $v$ and $\del v$. By the basic $L^{\infty}$ bounds, we see that
$$
|\del^IL^J \del_{\alpha}u|\leq CC_1\vep s^{-1+\delta},\quad |\del^IL^Jv|\leq CC_1\vep t^{-1}s^{\delta}, \quad |I|+|J|\leq N-2.
$$
Thus the desired bound on $T$ is established.
\end{proof}
%\marginpar{proof verified}

\begin{lemma}\label{lem 2 KG-infty}
Under the bootstrap assumption \eqref{eq 1 bootstrap}, the following bound holds for $|I|+|J|\leq N-4$:
\begin{equation}\label{eq 1 lem 2 KG-infty}
|S_1[\del^IL^J w]|\leq CC_1\vep (s/t)s^{-3+2\delta}.
\end{equation}
\end{lemma}
\begin{proof}
We recall that
$$
w = v - P^{\alpha\beta}\del_{\alpha}u\del_{\beta}u
$$
and we see that:
$$
S_1[\del^IL^J v] = \sum_a\delu_a\delu_a\del^IL^J v  + \frac{x^ax^b}{s^2}\delu_a\delu_b \del^IL^J v + \frac{2x^a}{s^2}\delu_a \del^IL^J v.
$$
In the right-hand-side we remark the following identities:
$$
\delu_a\delu_b f = t^{-1}L_a\left(t^{-1}L_b\right)f = t^{-2}L_aL_bf -t^{-2}(x^a/t)L_b .
$$
Then combined with the basic decay bounds \eqref{ineq 3 Linfty-basic}, we see that
\begin{equation}\label{eq 1 proof lem 2 KG-infty}
|S_1[\del^IL^J v]| \leq CC_1\vep (s/t)s^{-3+\delta}.
\end{equation}
%\marginpar{Verified}

Now we turn to the term $S_1[\del^IL^J\left(P^{\alpha\beta}\del_{\alpha}u\del_{\beta}u\right)]$.  Remark that for $|I|+|J|\leq N-2$,
\begin{equation}\label{eq 2 proof lem 2 KG-infty}
\big|\del^IL^J\big(\del_{\alpha}u\del_{\beta}u\big)\big|\leq C(C_1\vep)^2 s^{-2+2\delta}.
\end{equation}
Then we substitute this bound into the expression of $S_1$,
$$
\left|S_1[\del^IL^J \left(P^{\alpha\beta}\del_{\alpha}u\del_{\beta}u\right)]\right|\leq CC_1\vep s^{-4+2\delta}\leq CC_1\vep (s/t)s^{-3+2\delta}
$$
where we have applied the fact that
$$
s^{-1}\leq C(s/t).
$$

Then the desired result is established.
\end{proof}

Now we substitute the above bounds into \eqref{eq 6' KG-linfty}, also recall that for $|I|+|J|\leq N-3$,
$$
|L_a\del^IL^J v|\leq CC_1\vep t^{-1}s^{\delta}\leq CC_1\vep (s/t)^{2-3\delta},\quad  |F_{t,x}(\lambda)|\leq CC_1\vep (s/t)\lambda^{-2+3\delta}.
$$
then for $|I|+|J|\leq N-4$ %\marginpar{$|C_1\vep|\leq 1$}
\begin{equation}\label{eq 7 KG-linfty}
|\del^IL^Jw(t,x)|\leq CC_1\vep (s/t)^{2-3\delta}s^{-1},\quad |\del_t\del^IL^J w(t,x)|\leq CC_1\vep (s/t)^{1-3\delta}s^{-1}.
\end{equation}
Based on the bounds of $w$, we deduce the bounds on $v$. This is by recall that $w = v - P^{\alpha\beta}\del_{\alpha}u\del_{\beta}u$ and we recalling \eqref{ineq 1 Linfty-null}, then we obtain that for $|I|+|J|\leq N-4$:
%\marginpar{Bounds on $P^{\alpha\beta}$}
\begin{equation}\label{eq 8 KG-linfty}
|\del^IL^Jv(t,x)|\leq CC_1\vep (s/t)^{2-3\delta}s^{-1},\quad |\del_t\del^IL^J v(t,x)|\leq CC_1\vep (s/t)^{1-3\delta}s^{-1}.
\end{equation}

We remark that $\del_a \del^IL^Jv = \delu_a \del^IL^Jv - \frac{x^a}{t}\del_t\del^IL^Jv$, and by the fast decay of $\delu_a \del^IL^Jv = t^{-1}L_a\del^IL^J v $, we see that for $|I|+|J|\leq N-4$
\begin{equation}\label{eq 9 KG-linfty}
|\del_{\alpha}\del^IL^J v|\leq CC_1\vep (s/t)^{1-3\delta}s^{-1}.
\end{equation}
Then by estimates of commutators \eqref{ineq 1.1 comm}, we see that for $|I|+|J|\leq N-4$,
\begin{equation}\label{eq 10 KG-linfty}
|\del^IL^J \del_{\alpha}v| \leq CC_1\vep (s/t)^{1-3\delta}s^{-1}.
\end{equation}

In the same manner, for $|I|+|J|\leq N-5$,
\begin{equation}\label{eq 11 KG-linfty}
|\del_{\alpha}\del^IL^J v| + |\del^IL^J\del_{\alpha}v|\leq CC_1\vep (s/t)^{2-3\delta}s^{-1}.
\end{equation}

%\marginpar{Till here 12-04-2017 18:32}

\section{$L^{\infty}$ estimate on wave component}\label{sec w-infty}
The objective of this section is to prove the following bounds:
\begin{lemma}\label{lem 4 w-Kirchhoff}
Under the assumption of \eqref{eq 1 bootstrap}, the following bounds hold for $|I|+|J|\leq N-6$:
\begin{equation}\label{eq 1 lem 4 w-Kirchhoff}
|\delu_a\delu_a \del^IL^J u|\leq CC_1\vep(s/t)t^{-2}.
\end{equation}
\end{lemma}
\begin{proof}
As explained in introduction,
$$
\delu_a\delu_a\del^IL^J u = t^{-2}L_aL_a\del^IL^J u - (x^a/t^3)L_a\del^IL^Ju
$$
and thus (by the estimates of commutators) we need to establish the following bounds for $|I|+|J|\leq N-6$:
$$
|\del^IL^JL_aL_au|\leq CC_1\vep(s/t),\quad |\del^IL^JL_au|\leq CC_1\vep(s/t).
$$
And these bounds can be deduced by
\begin{equation}\label{eq 1 proof lem 4 w-Kirchhoff}
|\del^IL^J u|\leq CC_1\vep (s/t),\quad |I|+|J|\leq N-4.
\end{equation}
To prove this bound, we write the wave equation satisfied by $\del^IL^J u$:
\begin{equation}\label{eq 2 proof lem 4 w-Kirchhoff}
\Box \del^IL^J u = \del^IL^J \left(f(v,\del v)\right),\quad \big|\del_t\del^IL^Ju|_{\Hcal_2}\big|, \big|\del^IL^J u|_{\Hcal_2}\big|\sim CC_0\vep
\end{equation}
where $f(v,\del v) = v\left(A^{\alpha}\del_{\alpha}v + Rv\right) + Q^{\alpha\beta}\del_{\alpha}v\del_{\beta}v$.

We derive $f$ with respect to $\del^IL^J$, this gives a linear combination of bilinear terms composed by
$$
\del^{I_1}L^{J_1}v\del^{I_2}L^{J_2}v,\quad \del^{I_1}L^{J_1} \del_{\alpha}v\del^{I_2}L^{J_2}v,\quad \del^{I_1}L^{J_1}\del_{\alpha}v\del^{I_2}L^{J_2}\del_{\beta}v
$$
with $|I_1|+|I_2|+|J_1|+|J_2|\leq N-4$. Remark that when $N\geq 5$, either $|I_1|+|J_1|\leq N-5$ or $|I_2|+|J_2|\leq N-5$. Then by \eqref{eq 8 KG-linfty}, \eqref{eq 9 KG-linfty}, \eqref{eq 10 KG-linfty} and \eqref{eq 11 KG-linfty}, the following bound holds:
\begin{equation}\label{eq 3 proof lem 4 w-Kirchhoff}
|\del^IL^J\left(f(v,\del v)\right)|\leq C(C_1\vep)^2(s/t)^{1-6\delta}t^{-2}.
\end{equation}
Then by proposition \ref{prop 1 w-Kirchhoff}, \eqref{eq 1 lem 4 w-Kirchhoff} is established.
\end{proof}
%\marginpar{proof verified}

\section{Sharp decay estimate on $\del \del^IL^J u$}\label{sec sharp-dw}
\subsection{Algebraic preparation}
In this section we will establish the following bounds:
\begin{equation}\label{eq 1 dw-sharp}
|\del_{\alpha}\del^IL^Ju|\leq CC_1\vep s^{-1},\quad |\delu_a\del^IL^J u|\leq CC_1\vep t^{-1}, \quad \text{ for } |I|+|J|\leq N-6.
\end{equation}
Then by the estimates of commutators,
\begin{equation}
|\del^IL^J\del_{\alpha}u|\leq CC_1\vep s^{-1},\quad |\del^IL^J \delu_a u|\leq CC_1\vep t^{-1},\quad \text{ for } |I|+|J|\leq N-6.
\end{equation}
The bound on $\delu_a\del^IL^J u$ can be established as following. We see that by \eqref{ineq 2.1 comm},
$$
|\delu_a\del^IL^Ju|\leq  Ct^{-1}\sum_{|I|+|J|\leq N-5}|\del^IL^J u|.
$$
%\marginpar{Lack of a commutator $L_a,\del^IL^J$}
then by \eqref{eq 1 proof lem 4 w-Kirchhoff}, we see that for $|I|+|J|\leq N-5$,
\begin{equation}\label{eq 1.1 dw-sharp}
|\delu_a\del^IL^Ju|\leq CC_1\vep (s/t)t^{-1}.
\end{equation}
Then again by the estimates of commutators \eqref{ineq 2.1 comm},
\begin{equation}\label{eq 1.2 dw-sharp}
|\del^IL^J\delu_a u|\leq CC_1\vep(s/t)t^{-1}.
\end{equation}
Once $\delu_a\del^IL^Ju$ is bounded, by the relation $\del_a = \delu_a - (x^a/t)\del_t$, we see that the  estimates on $\del_a\del^IL^J u$ is reduced into the estimate on $\del_t\del^IL^Ju$. Thus in the following subsection we concentrate on this temporal derivative.

\subsection{Estimate on $\del_t\del^IL^Ju$}
Now we focus on the bound on $\del_t\del^IL^Ju$. This is by proposition \ref{prop 1 dw}. We derive the wave equation satisfied by $u$ with respect to $\del^IL^J$:
$$
\Box\del^IL^J u = \del^IL^J\left(f(v,\del v)\right)
$$
We recall that by \eqref{eq 1 lem 4 w-Kirchhoff} (with the notation of proposition \ref{prop 1 dw}), for $|I|+|J|\leq N-6$ (with $\delta\leq 1/10$),
$$
\left|\sum_a\delu_a\delu_a\del^IL^J u\right|\leq CC_1\vep (s/t)t^{-2}.
$$
Furthermore, for $|\del^IL^J \Box u| = \big|\del^IL^J\big(f(v,\del v)\big)\big|$, we recall that when $|I|+|J|\leq N-6$, this term is a linear combination of quadratic terms composed by $\del^IL^J v$ and $\del^IL^J\del_{\alpha}v$. Then by \eqref{eq 8 KG-linfty} and \eqref{eq 11 KG-linfty},
$$
\big|\del^IL^J \Box u\big|\leq C(C_1\vep)^2(s/t)^{4-6\delta}s^{-2}\leq C(C_1\vep)^2(s/t)t^{-2}.
$$

Substitute the above bounds into \eqref{eq 1 prop-main Linfty-dw}, we see that
\begin{equation}\label{eq 1 dw}
|R_w(\tau;t,x)| \leq CC_1\vep (s/t)^2 t^{-1}\leq CC_1\vep \frac{t-r}{t^2}.
\end{equation}
%\marginpar{$|C_1\vep|\leq 1$}
So by proposition \ref{prop 1 dw} (with the notation in this proposition),
$$
|\del_t\del^IL^J u|\leq CC_0s^{-1} + s^{-1}CC_1\vep\int_2^te^{-\int_{\tau}^t(t-r)t^{-2}|_{\gamma(\eta;t,x)}d\eta}\frac{t-r}{t^2}\bigg|_{\gamma(\tau;t,x)}d\tau.
$$
We denote by $\xi(\tau)$ the function $\frac{t-r}{t^2}\bigg|_{\gamma(\tau;t,x)}$. And the above inequality becomes:
$$
|\del_t\del^IL^J u|\leq CC_0s^{-1} + s^{-1}CC_1\vep\int_2^te^{-\int_{\tau}^t\xi(\eta)d\eta}\xi(\tau)d\tau\leq CC_0\vep s^{-1} + CC_1\vep s^{-1}.
$$
We recall that in the bootstrap assumption we have chosen that $C_1>C_0$. So we conclude that
\begin{equation}\label{eq 2 dw}
|\del_t\del^IL^J u|\leq CC_1\vep s^{-1}, \quad |I|+|J|\leq N-6.
\end{equation}
By the argument in the last subsection, we see that
\begin{equation}\label{eq 4 dw}
|\del_a\del^IL^J u|\leq CC_1\vep s^{-1}, \quad |I|+|J|\leq N-6.
\end{equation}
By the estimate of commutators, we see that
\begin{equation}\label{eq 3 dw}
|\del^IL^J\del_t u|+ |\del^IL^J\del_a u|\leq CC_1\vep s^{-1}, \quad |I|+|J|\leq N-6.
\end{equation}
%==========================================================================================================================================================

\section{Refined energy estimate and conclusion}\label{sec energy conclusion}

Now we are about to prove \eqref{eq 1 bootstrap}. This is based on the energy estimate Proposition \ref{prop energy} and the sharp $L^{\infty}$ bounds on $u$ and $v$, more precisely, on \eqref{eq 1.2 dw-sharp}, \eqref{eq 3 dw}, \eqref{eq 8 KG-linfty}, \eqref{eq 9 KG-linfty}, \eqref{eq 10 KG-linfty} and \eqref{eq 11 KG-linfty}.

For the wave component, we derive the wave equation with respect to $\del^IL^J$ with $|I|+|J|\leq N$ and see that
\begin{equation}\label{eq 1 energy-refined}
\Box \del^IL^Ju = \del^IL^Jf
\end{equation}
with
$$
f(v,\del v) : = v\left(A_1^{\alpha}\del_{\alpha}v + R v\right) + Q^{\alpha\beta}\del_{\alpha}v\del_{\beta}v
$$
We establish the following $L^2$ bound on the source term:
\begin{lemma}\label{lem 1 energy-refined}
Under the bootstrap assumption, the following bound holds:
\begin{equation}\label{eq 2 energy-refined}
\|\del^IL^J f\|_{L^2(\Hcal_s)}\leq C(C_1\vep)^2s^{-1+\delta}, \quad 2\leq s\leq s_0, \quad |I|+|J|\leq N.
\end{equation}
\end{lemma}
\begin{proof}
This is by direct calculation. We remark that $f$ is a linear combination of the quadratic form on $v$ and $\del v$. We take the term $\del_\alpha v\del_\beta v$ as an example:
$$
\del^IL^J \left(\del_{\alpha}v\del_{\beta}v\right) = \sum_{I_1+I_2=I\atop J_1+J_2=J}\del^{I_1}L^{J_1}\del_{\alpha}v\del^{I_2}L^{J_2}\del_{\beta}v.
$$
And we see that $|I_1|+|J_1|\leq [N/2]$ or $|I_2|+|J_2|\leq [N/2]$. Thus when $N\geq 11$, we see that at least one satisfies $|I_i|+|J_i|\leq N-6$. Without lose of generality, we suppose that $|I_1|+|J_1|\leq N-6$, thus by applying \eqref{eq 11 KG-linfty} and \eqref{eq 1 L2-basic} (recall that $\delta<1/10$)
$$
\aligned
\|\del^{I_1}L^{J_1}v\del^{I_2}L^{L_2}v\|_{L^2(\Hcal_s)} \leq& CC_1\vep\|(s/t)^{2-3\delta}s^{-1}\del^{I_2}L^{L_2}v\|_{L^2(\Hcal_s)}
\\
\leq& CC_1\vep s^{-1}\|(s/t)\del^{I_2}L^{J_2}v\|_{L^2}\leq C(C_1\vep)^2s^{-1+\delta}.
\endaligned
$$
The other terms of $f$ are treated in the same manner and we omit the detail.
\end{proof}

For the Klein-Gordon component, we derive the Klein-Gordon equation with respect to $\del^IL^J$ and see that
\begin{equation}\label{eq 3 energy-refined}
\Box \del^IL^J v + \del^IL^J v = \del^IL^J \left(P^{\alpha\beta}\del_{\alpha}u\del_{\beta}u\right)
\end{equation}
Then we establish the following bound on the source of the Klein-Gordon equation:
\begin{lemma}\label{lem 2 energy-refined}
Under the bootstrap assumption, the following bound holds:
\begin{equation}\label{eq 4 energy-refined}
\|\del^IL^J\left(P^{\alpha\beta}\del_{\alpha}u\del_{\beta}u\right)\|_{L^2(\Hcal_s)}\leq C(C_1\vep)^2 s^{-1+\delta}, \quad 2\leq s\leq s_0,\quad |I|+|J|\leq N.
\end{equation}
\end{lemma}
\begin{proof}
To prove \eqref{eq 4 energy-refined} we need to evoke the null structure of $P^{\alpha\beta}$. We see that
\begin{equation}\label{eq 5 energy-refined}
\aligned
P^{\alpha\beta}\del_{\alpha}u\del_{\beta}u =& \Pu^{00}\del_tu\del_tu + \Pu^{a0}\delu_au\del_tu + \Pu^{0a}\del_tu\delu_au + \Pu^{ab}\delu_au\delu_bu
\endaligned
\end{equation}

For the the first term in right-hand-side of \eqref{eq 5 energy-refined}, we see that by \eqref{eq 1 null},
$$
\aligned
\|\del^IL^J\left(\Pu^{00}\del_tu\del_tu\right)\|_{L^2(\Hcal_s)}
\leq &\sum_{I_1+I_2+I_3=I\atop J_1+J_2+J_3=J}\|\del^{I_3}L^{J_3}\Pu^{00}\cdot \del^{I_1}L^{J_1}\del_tu\cdot \del^{I_2}L^{J_2}\del_t u\|_{L^2(\Hcal_s)}
\\
\leq& C\sum_{|I_1|+|I_2|\leq|I|\atop|J_1|+|J_2|\leq |J| }\|(s/t)^2\del^{I_1}L^{J_1}\del_tu\cdot \del^{I_2}L^{J_2}\del_t u\|_{L^2(\Hcal_s)}.
\endaligned
$$
Then by the same argument in the proof of lemma \ref{lem 1 energy-refined}, we suppose without lose of generality that $|I_1|+|I_2|\leq N-6$. Then by \eqref{eq 4 dw},
$$
\|(s/t)^2\del^{I_1}L^{J_1}\del_tu\cdot \del^{I_2}L^{J_2}\del_t u\|_{L^2(\Hcal_s)}\leq CC_1\vep s^{-1}\|(s/t)^2\del^{I_2}L^{J_2}\del_t u\|_{L^2(\Hcal_s)}
\leq C(C_1\vep)^2s^{-1+\delta}.
$$

For the rest terms in right-hand-side of \eqref{eq 5 energy-refined}, we remark that each contains at least one ``good'' derivative (i.e. $\delu_a$). We take the term $\delu_au\del_tu$ as an example:
$$
\aligned
\|\del^IL^J\left(\Pu^{a0}\delu_au\del_tu\right)\|_{L^2{\Hcal_s}}
\leq & C\sum_{|I_1|+|I_2|\leq |I|\atop|J_1|+|J_2|\leq |J|}\|\del^{I_1}L^{J_1}\delu_au\del^{I_2}L^{J_2}\del_tu\|_{L^2(\Hcal_s)}
\endaligned
$$
where we applied the fact that $\Pu^{a0}$ is homogeneous of degree zero. Then we see that when $|I_1|+|J_1|\leq N-6$, by \eqref{eq 1.1 dw-sharp}
$$
\aligned
\|\del^{I_1}L^{J_1}\delu_au\del^{I_2}L^{J_2}\del_tu\|_{L^2(\Hcal_s)}\leq& CC_1\vep\|t^{-1}\del^{I_2}L^{J_2}\del_tu\|_{L^2(\Hcal_s)}
\\
=&CC_1\vep s^{-1}\|(s/t)\del^{I_2}L^{J_2}\del_tu\|_{L^2(\Hcal_s)}\leq C(C_1\vep)^2s^{-1+\delta}.
\endaligned
$$
When $|I_1|+|J_1|> N-6$, we see that $|I_2|+|J_2|\leq 5\leq N-6$ thus
$$
\aligned
\|\del^{I_1}L^{J_1}\delu_au\del^{I_2}L^{J_2}\del_tu\|_{L^2(\Hcal_s)}\leq& CC_1\vep s^{-1}\|\del^{I_1}L^{J_1}\delu_a u\|_{L^2(\Hcal_s)}
\leq C(C_1\vep)^2 s^{-1+\delta}.
\endaligned
$$
The rest terms such as $\delu_au\delu_bu$, are easier than the above two cases and we omit the detail.
\end{proof}

Now we are ready to establish the refined energy bounds \eqref{eq 2 bootstrap}. This is by proposition \ref{prop energy}. We make the following remark. When
$$
\|\del^IL^J u_0\|_{L^2(\Hcal_2)} + \|\del^IL^J v_0\|_{L^2(\Hcal_2)}\leq \vep, \quad |I|+|J|\leq N,
$$
$$
\|\del^{I'}L^{J'} u_1\|_{L^2(\Hcal_2)} + \|\del^{I'}L^{J'} v_1\|_{L^2(\Hcal_2)}\leq \vep, \quad |I'|+|J'|\leq N-1,
$$
Then there exists a constant $C_0$, determined by $N$ such that
$$
E(\del^IL^Ju,2)^{1/2} + E_1(\del^IL^J v,2)^{1/2}\leq C_0\vep.
$$

For the wave component, we recall \eqref{eq 1 energy-refined} and \eqref{eq 1 energy} combined with lemma \ref{lem 1 energy-refined}:
$$
E(\del^IL^Ju,s)^{1/2}\leq E(\del^IL^Ju,2)^{1/2} + \int_2^s\|\del^IL^J f\|_{\Hcal_\tau}d\tau\leq C_0\vep + C(C_1\vep)^2s^{\delta}
$$
and for Klein-Gordon component,
$$
\aligned
E_1(\del^IL^Jv,s)^{1/2}\leq& E_1(\del^IL^Jv,2)^{1/2} + \int_2^s\|\del^IL^J\left(P^{\alpha\beta}\del_{\alpha}u\del_{\beta}u\right)\|_{L^2(\Hcal_\tau)}d\tau
\\
\leq& C_0\vep + C(C_1\vep)^2 s^{\delta}.
\endaligned
$$

Now we make the following choice of $(C_1,\vep)$:
$$
C_1 > 2C_0, \quad \vep < \frac{C_1-2C_0}{2CC_1},\quad  C_1\vep\leq 1.
$$
This leads to \eqref{eq 2 bootstrap}. Then we conclude by the desired result of theorem \ref{thm main}.

\end{document}